\documentclass[12pt]{amsart}

\usepackage{fancybox,color}

\newcommand{\kom}[1]{}

\renewcommand{\kom}[1]{{\bf [#1]}}

\definecolor{gruen}{cmyk}{1.0,0.2,0.7,0.07}

\usepackage{latexsym}
\usepackage{amssymb}
\usepackage{mathrsfs}
\usepackage{amsmath}
\usepackage{fancybox,color,graphicx}
\usepackage{enumerate}
\usepackage[active]{srcltx}
\usepackage[colorlinks]{hyperref}
\usepackage{hypernat}
\usepackage{setspace}
 \newcommand{\eps}{{\varepsilon}}
 \newcommand{\osc}{\mbox{osc}}
 \def\1{\raisebox{2pt}{\rm{$\chi$}}}
 
\newcommand{\abs}[1]{\left|#1\right|}

\def\vint_#1{\mathchoice%
          {\mathop{\kern 0.2em\vrule width 0.6em height 0.69678ex depth -0.58065ex
                  \kern -0.8em \intop}\nolimits_{\kern -0.4em#1}}%
          {\mathop{\kern 0.1em\vrule width 0.5em height 0.69678ex depth -0.60387ex
                  \kern -0.6em \intop}\nolimits_{#1}}%
          {\mathop{\kern 0.1em\vrule width 0.5em height 0.69678ex
              depth -0.60387ex
                  \kern -0.6em \intop}\nolimits_{#1}}%
          {\mathop{\kern 0.1em\vrule width 0.5em height 0.69678ex depth -0.60387ex
                  \kern -0.6em \intop}\nolimits_{#1}}}
\def\vintslides_#1{\mathchoice%
          {\mathop{\kern 0.1em\vrule width 0.5em height 0.697ex depth -0.581ex
                  \kern -0.6em \intop}\nolimits_{\kern -0.4em#1}}%
          {\mathop{\kern 0.1em\vrule width 0.3em height 0.697ex depth -0.604ex
                  \kern -0.4em \intop}\nolimits_{#1}}%
          {\mathop{\kern 0.1em\vrule width 0.3em height 0.697ex depth -0.604ex
                  \kern -0.4em \intop}\nolimits_{#1}}%
          {\mathop{\kern 0.1em\vrule width 0.3em height 0.697ex depth -0.604ex
                  \kern -0.4em \intop}\nolimits_{#1}}}

\newcommand{\kint}{\vint}
\newcommand{\intav}{\vint}
\newcommand{\aveint}[2]{\mathchoice%
          {\mathop{\kern 0.2em\vrule width 0.6em height 0.69678ex depth -0.58065ex
                  \kern -0.8em \intop}\nolimits_{\kern -0.45em#1}^{#2}}%
          {\mathop{\kern 0.1em\vrule width 0.5em height 0.69678ex depth -0.60387ex
                  \kern -0.6em \intop}\nolimits_{#1}^{#2}}%
          {\mathop{\kern 0.1em\vrule width 0.5em height 0.69678ex depth -0.60387ex
                  \kern -0.6em \intop}\nolimits_{#1}^{#2}}%
          {\mathop{\kern 0.1em\vrule width 0.5em height 0.69678ex depth -0.60387ex
                  \kern -0.6em \intop}\nolimits_{#1}^{#2}}}

\newtheorem{theorem}{Theorem}[section]
\newtheorem{corollary}[theorem]{Corollary}
\newtheorem{lemma}[theorem]{Lemma}
\newtheorem{propo}[theorem]{Proposition}
\newtheorem{definition}[theorem]{Definition}
\newtheorem{remark}[theorem]{Remark}

 \renewcommand{\div}{\operatorname{div}}
 
  \newcommand{\R}{\mathbb{R}}

\newcommand{\xintloo}[1]{\int\limits_{#1} \kern-18pt\raise4pt\hbox to7pt {\hrulefill}\ }   \numberwithin{equation}{section}

\newcommand{\G}{{\mathcal G}}
 
  \newcommand{\Z}{\mathbb{Z}}

 \newcommand{\ud}{\, d}

\newcommand{\I}{\textrm{I}}

\newcommand\vv{\overline{v}}

\newcommand{\B}{\textup{B}}
\newcommand{\USC}{\textup{USC}}
\newcommand{\LSC}{\textup{LSC}}

\begin{document}
 \title[]{Convergence of dynamic programming principles \\ for the $p$-Laplacian}
 
  \author[F. del Teso]{F\'elix del Teso}
 \address{F\'elix del Teso
 \hfill\break\indent
Department of Mathematical Analysis and Applied Mathematics
\hfill\break\indent
Universidad Complutense de Madrid
\hfill\break\indent 28040 Madrid,  Spain
\hfill\break\indent
{\tt fdelteso@ucm.es}}
 
 \author[J. Manfredi]{Juan J. Manfredi}
 \address{Juan J. Manfredi
 \hfill\break\indent
Department of Mathematics
\hfill\break\indent
University of Pittsburgh
\hfill\break\indent Pittsburgh, PA 15260, USA
\hfill\break\indent
{\tt manfredi@pitt.edu}}

 \author[M. Parviainen]{Mikko Parviainen}
  \address{Mikko Parviainen
 \hfill\break\indent
Department of Mathematics and Statistic
\hfill\break\indent
University of Jyv\"{a}skyl\"{a}
\hfill\break\indent 40014 Jyv\"{a}skyl\"{a}, Finland.
\hfill\break\indent
{\tt mikko.j.parviainen@jyu.fi}}
 
  \date{\today}
\maketitle

\begin{abstract}
We provide a unified strategy to show that solutions of  dynamic programming principles associated to the $p$-Laplacian converge to the solution of the corresponding  Dirichlet problem. Our approach includes all previously known cases for continuous and discrete dynamic programming principles, provides new results, and gives a convergence proof free of probability arguments.
\end{abstract}
{\singlespacing
\tableofcontents        
}

\section{Introduction}
Consider a bounded Lipschitz domain $\Omega\subset\mathbb{R}^{d}$ and a  Lipschitz  function $g\colon\partial\Omega\mapsto\mathbb{R}$. We are interested in the Dirichlet problem
\begin{equation}\label{dproblem:intro}  
\left\{
\begin{array}{cccl}
-\Delta_{p} u & = & 0& \text{ in } \Omega\\
u &=& g & \text{ on } \partial{\Omega},
\end{array}
\right.
\end{equation}
where $\Delta_{p}u=\div\left(|\nabla u|^{p-2}\nabla u\right)$ is the $p$-Laplace operator. 
Given a small $\eps>0$ consider a thin strip around the boundary $\Gamma_{\eps}=\{ x\in \R^d\setminus\Omega\colon d(x, \partial\Omega)\le \eps\}$ and extend $g$ to $\Gamma_{\eps}$. For a representative example choose $\alpha$ and $\beta$ non-negative such that $\alpha+\beta=1$ and a function $v\colon\Omega\cup\Gamma_{\eps}\mapsto\mathbb{R}$, consider the dynamic programming principle
\begin{equation}\label{pharmonious:intro}  
\left\{
\begin{array}{cccl}
 v(x)& = &  \displaystyle\frac{\alpha}{2}\left(\displaystyle\sup_{B_{\eps}(x)}  v+\inf_{B_{\eps}(x)}  v \right)+ \beta \displaystyle \intav_{B_{\eps}(x)} v(y)\, dy& \text{ in } \Omega\\
v(x) &=& g(x)\hspace{6cm}  & \text{ on } \Gamma_{\eps}.
\end{array}
\right.
\end{equation}
When $p\ge 2$, $\alpha=\frac{p-2}{p+d}$, and $\beta=\frac{2+d}{p+d}$, it turns out that there is always a unique solution to \eqref{pharmonious:intro} that we label $u_{\eps}$ and call $(\eps,p)$-harmonious. 
Manfredi-Parviainen-Rossi proved in \cite{MPR12} that $u_{\eps}$ converges uniformly to $u$ the solution of
\eqref{dproblem:intro} when $\eps\to 0$. This theorem has been extended to the case $1<p<2$ with variable  $p(x)$  (\cite{AHP17}), to the single and double obstacle problems (\cite{LM17,CLM17}), and to the Heisenberg group (\cite{LMR18}). All of these proofs use  probability tools based on the fact that the value function of a tug-of-war game with noise  converges to a $p$-harmonic function \cite{PSSW09, PS08}. The idea is to obtain approximative uniform  continuity close to the boundary by using a barrier argument, and then "copying of strategies" to infer approximative uniform continuity also far away from the boundary.  Then the proof is completed using a version of the Arzel\`a-Ascoli theorem together with stability of viscosity solutions. If the setting is not translation invariant like in the case of $p(x)$, then copying of strategies can be replaced by local regularity arguments \cite{LP18}.
 
The existence and uniqueness of the solution $u_{\eps}$ to \eqref{pharmonious:intro}, as well as the fact that $u_{\eps}$ coincides with the value function of the tug-of-war with noise is proven in \cite{LPS14}. It also contains a modified dynamic programming principle with continuous solutions (Theorem 4.1), which has turned out to be useful in extending the existence and uniqueness to the case $1<p<2$ \cite{Har16}, and also in \cite{AHP17} mentioned above.

Heuristically the link between \eqref{dproblem:intro}  and \eqref{pharmonious:intro}   can be seen by using the classical Taylor expansion
\[
u(y)=u(x)+\nabla u(x)\cdot(y-x)
+\frac12 D^2u(x)(y-x)\cdot (y-x)+O(|y-x|^3)
\]
for smooth $u$ with nonvanishing gradient.
We average over $B_{\eps}(x)$ and obtain
\begin{equation}
\label{exp.lapla} u(x) -  \intav_{B_\eps(x)} u\ud y = -
  \frac{\eps^2}{2(n+2)}\Delta u (x)  + O(\eps^3),
\end{equation}
where
\[
\kint_{ B_{\eps}(x)} u \ud y=\frac{1}{\abs{B_{\eps}(x)}}\int_{ B_{\eps}(x)} u \ud y.
\]

The gradient direction is almost maximizing, and the opposite direction is almost minimizing. Thus, summing up the two Taylor expansions with these choices roughly gives us
\begin{equation}\label{exp.infty}
\begin{split}
u(x) &- \frac{1}{2} \left\{ \sup_{ B_{\eps}(x)} u
+ \inf_{ B_{\eps}(x)} u \right\} \\&\approx u(x)-
\frac{1}{2} \left\{ u \left(x+\eps\frac{\nabla u(x)}{\abs{\nabla
u(x)}} \right)+ u\left(x-\eps\frac{\nabla u(x)}{\abs{\nabla u(x)}}\right)\right\}\\
&= - \frac{\eps^2}{2} \Delta^N_\infty u (x) + O(\eps^3)
\end{split}
\end{equation}
since
$$
D^2u(x)\frac{\nabla u(x)}{\abs{\nabla u(x)}}\cdot \frac{\nabla u(x)}{\abs{\nabla u(x)}}=\left\langle D^2u(x)\frac{\nabla u(x)}{\abs{\nabla u(x)}}, \frac{\nabla u(x)}{\abs{\nabla u(x)}} \right\rangle =:\Delta^N_\infty u (x).
$$
Next we multiply \eqref{exp.lapla} and \eqref{exp.infty} by
suitable constants $\alpha, \beta\ge 0$, $\alpha+\beta=1$, and add
up the formulas so that we have the operator
$\Delta u+(p-2) \Delta^N_\infty u =\abs{\nabla u}^{2-p}\Delta_{p} u$
 on the right hand side in the resulting expansion. Thus for example if $u$ satisfies \eqref{dproblem:intro}, then $\abs{\nabla u}^{2-p}\Delta_{p} u=0$ and, dropping the error term in the expansion, we end up with  \eqref{pharmonious:intro}. 

\par
One of the main contributions of this paper is to present a unified and purely analytical proof of the convergence results for general dynamic programming principles associated to the Dirichlet problem \eqref{dproblem:intro}.
We do it in two steps. First, we consider the case of smooth domains ($\Omega$ is of class $C^{2}$) where we show that for the $p$-Laplacian with $p\in(1,\infty]$,  the concept of generalized viscosity solution of Barles and Ishii (\cite{BP87,BP88,I89}),  that takes into account the boundary condition, is equivalent to the classical notion of viscosity solution. This equivalence allows us to follow the scheme of Barles-Souganidis \cite{BS91} to show convergence. In the language of \cite{BS91} we show the that $p$-Laplacian satisfies the \lq\lq strong uniqueness property\rq\rq. 

For Lipschitz domains, we adapt a boundary regularity argument from  \cite{MPR12} showing the existence of good barriers near boundary points. These barriers are built on certain ring domains using the fundamental solution of the $p$-Laplacian with pole at the center of the ring, which are $C^{2}$-domains. 

\par
As mentioned above our method is quite general requiring only certain properties of the average operator that appears in the dynamic programming principle (see Section \S \ref{sec:gen}). As a  consequence of this generality we cover also discrete dynamic programming principles that when restricted to a lattice can be seen as convergent numerical schemes  for the Dirichlet problem \eqref{dproblem:intro}.

\medskip
\textbf{Organization of the paper:}\nopagebreak
\par
In Section \S \ref{sec:gen} we present the general framework with precise assumptions and definition and state the main results. Section \S \ref{sec:SUPC2domains} is devoted to the study of generalized viscosity solutions and the strong uniqueness property for the $p$-Laplacian. In Section \S \ref{sec:conv} we give an interpretation of dynamic programming principles as schemes and use this framework to prove the desired convergence. In Section \S \ref{sec:contDPP}, we extend our results to a class of dynamic programming principles with continuous solutions. Finally, in Section \S \ref{sec:examples} we present a series of examples included in our general theory.

\medskip
\textbf{Notation}
\par
$\bullet$ $\USC(\mathcal{S})$ $\equiv$ Upper semi-continuous function on the set $\mathcal{S}$

$\bullet$ $\LSC(\mathcal{S})$ $\equiv$ Upper semi-continuous function on the set $\mathcal{S}$

$\bullet$ $\B(\mathcal{S})$ $\equiv$ Bounded functions on the set $\mathcal{S}$.

$\bullet$ $C(\mathcal{S})$ $\equiv$ Continuous functions on the set $\mathcal{S}$.

$\bullet$ $C_b^k(\mathcal{S})$ $\equiv$ Functions in $\mathcal{S}$ with continuous bounded derivatives up to order $k$.

$\bullet$ $f_\eps=o(\eps^k)$ for some $k\geq0$ when  $|f_\eps|/\eps^{k} \to 0$ as $\eps\to0$ 

$\bullet$ $f_\eps=O(\eps^k)$ for some $k\geq0$ when  $|f_\eps|\leq C \eps^k$ for some constant independent of $\eps$.

\section{General framework and main results}\label{sec:gen}

\subsection{Dirichlet problem for the $p$-Laplacian}
In this section we revisit the concept of viscosity solution for the Dirichlet problem associated to the $p$-Laplacian. 

First, consider the \emph{$p$-Laplacian} operator defined for $p\in(1,\infty)$ and a smooth function $\phi$ as
\[
\Delta_p\phi(x):=\div\left(|\nabla \phi(x)|^{p-2}\nabla\phi(x)\right).
\]
It is also interesting to consider the so-called \emph{normalized $p$-Laplacian} defined whenever $\nabla\phi(x)\not=0$ as 
\[
\Delta_p^{N}\phi(x):=\frac{1}{p}\frac{1}{|\nabla \phi(x)|^{p-2}}\div\left(|\nabla \phi(x)|^{p-2}\nabla\phi(x)\right)=\frac{1}{p}\frac{1}{|\nabla \phi(x)|^{p-2}}\Delta_p \phi(x).
\]
The limit cases $p=1$ and $p=\infty$ are given by
\[
\Delta_1\phi(x)=\div\left(\frac{\nabla \phi(x)}{|\nabla \phi(x)|}\right), \qquad \Delta_1^N\phi(x)=|\nabla\phi(x)|\Delta_1\phi(x),
\]
and 
\[
\Delta_\infty\phi(x)=\langle D^2 \phi(x)\nabla\phi(x), \nabla \phi(x) \rangle, \qquad \Delta_\infty^N\phi(x)=\frac{1}{|\nabla \phi(x)|^2}\Delta_\infty\phi(x).
\]
\begin{remark}\label{rem:interpPlap}
We recall the following relations between the above defined operators valid for any $p\in(1,\infty)$:
\begin{equation}\label{eq:interp1}
\Delta_p^N\phi(x)= \frac{1}{p}\Delta_1^N\phi(x)+\frac{p-1}{p} \Delta_\infty^N \phi(x)
\end{equation}
and 
\begin{equation}\label{eq:interp2}
\Delta_p^N\phi(x)= \frac{1}{p}\Delta\phi(x)+\frac{p-2}{p} \Delta_\infty^N \phi(x)
\end{equation}
\end{remark}

Let $p\in(1,\infty]$ and $\Omega \subset \R^d$ be a Lipschitz bounded domain. Consider the Dirichlet problem given by 
\begin{equation}\label{dproblem} 
\left\{
\begin{array}{cccl}
-\Delta_{p} u & = & 0& \text{ in } \Omega\\
u &=& g & \text{ on } \partial{\Omega},
\end{array}
\right.
\end{equation}
for some $g\in C(\partial \Omega)$. Recall  the  classical concept of viscosity solutions for the Dirichlet problem \eqref{dproblem} (see eg. \cite{JLM01}).

\begin{definition}\label{def:ViscSub} Let $p\in (1,\infty]$. A function $u\in \USC(\overline{\Omega})$ is a \textbf{viscosity 
$p$-subsolution} of \eqref{dproblem} if whenever $x_{0}\in\overline{\Omega}$ and $\phi\in C^{2}(\overline{\Omega})$ satisfy
\begin{equation*} \left\{
\begin{array}{ccl}
\phi(x_{0})  & = & u(x_{0}),\\
\phi(x)  &>& u(x) \quad \text{for}\quad x\not=x_{0},\\
\nabla \phi(x_0)&\not=&0,
\end{array}
\right.
\end{equation*}
we have
\begin{align}
\label{eq:intC1} -\Delta_p\phi(x_0)&\leq0 \quad \text{ if } \quad  x_0\in \Omega,\\
\label{eq:BC1} u(x_0)-g(x_0)&\leq0 \quad \text{ if } \quad  x_0\in\partial\Omega.
\end{align}
\end{definition}

\begin{definition}\label{def:ViscSuper} Let $p\in (1,\infty]$. A function $u\in \LSC(\overline{\Omega})$ is a \textbf{viscosity 
$p$-supersolution} of \eqref{dproblem} if whenever $x_{0}\in\overline{\Omega}$ and $\phi\in C^{2}(\overline{\Omega})$ satisfy
\begin{equation*} \left\{
\begin{array}{ccl}
\phi(x_{0})  & = & u(x_{0}),\\
\phi(x)  &< & u(x) \quad \text{for}\quad x\not=x_{0},\\
\nabla \phi(x_0)&\not=&0
\end{array}
\right.
\end{equation*}
we have
\begin{align}
\label{eq:intC2} -\Delta_p\phi(x_0)&\geq0 \quad \text{ if } \quad  x_0\in \Omega,\\
\label{eq:BC2} u(x_0)-g(x_0)&\geq0 \quad \text{ if } \quad  x_0\in\partial\Omega.
\end{align}
\end{definition}

\begin{definition}\label{def:ViscSol}
Let $p\in (1,\infty]$. A function $u\in C(\overline{\Omega})$ is a \textbf{viscosity 
$p$-solution} of  \eqref{dproblem} if it is both viscosity $p$-subsolution and $p$-supersolution.

\end{definition}
Observe that we only use test functions with non-vanishing gradient.
Existence of weak solutions of the Dirichlet problem \eqref{dproblem} when $p\in(1,\infty)$, $\Omega$ is a bounded Lipschitz domain,  and $g\in C(\partial\Omega)$ is classical, see \cite{HKM06}. The fact that weak solutions are viscosity solutions, as well as many other properties, is established in \cite{JLM01}. Uniqueness is a folklore result that can be found in appendix 2 of
 \cite{LM17}.
 
\subsection{Mean value properties and dynamic programming principles}
Given a bounded domain $\Omega$ we define the following sets:

\emph{Outer boundary strip:} $\Gamma_\eps:=\{x\in\R^d\setminus\Omega \colon d(x,\partial \Omega)\leq\eps\}$ and $O:=\Gamma_1$.

\emph{Inned boundary strips:} $I_\eps:=\{x\in\Omega \colon d(x,\partial \Omega)\leq \eps\}$ and $I:=I_1$.

\emph{Extended domain:} $\Omega_E:= \Omega\cup O$ (note that $\Omega_E$  is a closed set).

Two basic elements in the construction of a dynamic programming principle are the concepts of averages and mean value properties.
\begin{definition}
We say that an operator $A: \B(\Omega_E)\to \R$ is an \textbf{average} if the following assumptions hold:
\begin{enumerate}
\item (Stable) $\displaystyle\inf_{y\in\Omega_E} \phi(y) \leq A[\phi](x) \leq \sup_{y\in\Omega_E}  \phi(y)$ for all $x\in \Omega$
\item (Monotone) If $\phi\leq \psi$ in $\Omega_E$ then $A[\phi]\leq A[\psi]$ in $\Omega_E$.
\medskip
\item (Affine invariance) $A[\lambda \phi+\xi]=\lambda A[\phi]+\xi$ for every $\xi\in \R$ and $\lambda>0$.
\end{enumerate}
\end{definition}
We refer to Section \ref{sec:examples} for concrete examples of averages that are closely connected to the the Laplacian,  $\infty$-Laplacian, and the $p$-Laplacian. Among others, they can take integral form like in Section \ref{sec:MVPp2}, $\sup$/$\inf$ forms as in Section \ref{sec:MVPinfty}, discrete versions of them like in Section \ref{sec:MVPdiscrete}, as well as finite difference forms as in Section \ref{sec:MVPFD}. 

\begin{definition}\label{def:MVP}
We say that family of averages $\{A_\eps\}_{\eps>0}$ is an  \textbf{(asymptotic) mean value property} for the $p$-Laplacian if for every $\phi\in C^\infty_b(\Omega_E)$ such that $\nabla \phi \not=0$ we have
\begin{equation}\label{eq:MVP}
\phi(x)= A_\eps[\phi](x)+c_{p,d}\eps^2\left(-\Delta_p^N \phi(x)\right)+ o(\eps^2).
\end{equation}
for some constant $c_{p,d}>0$ and where the constant in $o(\eps^2)$ can be taken uniformly for all $x\in \overline{\Omega}$.
\end{definition}

 For examples of mean value properties, we refer the reader again to Section \ref{sec:examples}. 

\begin{remark} An alternative to \eqref{eq:MVP} is given by 
\[
\left\|\frac{\phi-A_\eps[\phi]}{c_{p,d} \eps^2}-\left(-\Delta_p^N \phi\right)\right\|_{L^\infty(\overline{\Omega})}= o(1)
\textrm{ as } \eps\to0.
\]
\end{remark}
\medskip 
Given $g\in C(\partial \Omega)$ let $G$ to be a continuous extension of $g$ to $\Omega_E$. Associated to a mean value property $A_\eps$ we have a dynamic programming principle given by
\begin{equation}\label{eq:expDPP1}
\left\{
\begin{array}{cccl}
u_\eps(x) & = & A_\eps [u_\eps] (x) & \text{ in } \Omega,\\
u_\eps(x)&=& G(x) & \text{ on } {O}.
\end{array}
\right.
\end{equation}

Our first general result (Theorem \ref{thm:main}) concerning $C^2$-domains, will require the following assumption 
\begin{equation}\label{as:stab}
\begin{split}
&\textup{For all $ \eps>0$ there exists $ u_\eps \in B(\Omega_E)$, a  solution of \eqref{eq:expDPP1}
with a}\\ &\textup{bound on $\|u_\eps\|_{L^\infty(\Omega_E)}$ uniform in $\eps$.}
\end{split}
\end{equation}
\begin{remark}
Note that the uniform bound in  the assumption \eqref{as:stab} follows trivially if the dynamic programming principle \eqref{eq:expDPP1} satisfies  the standard maximum principle given by
\[
\|u_\eps\|_{L^\infty(\Omega_E)}\leq \|G\|_{L^\infty(O)}\qquad \textrm{for all } \eps>0.
\]
\end{remark}

Our second general result (Theorem \ref{thm:main2}) for Lipschitz domains requires slightly more restrictive (but still very natural) assumptions on \eqref{eq:expDPP1}:
\begin{equation}\label{as:stab2}
\begin{split}
&\textup{For all $ \eps>0$ there exists $ u_\eps \in B(\Omega_E)$, a  solution of \eqref{eq:expDPP1},  and }\\
&\inf_{O} G \leq u_\eps(x) \leq \sup_{O}{G} \quad \textup{for all}\quad x\in \Omega.
\end{split}
\end{equation}
\begin{equation}\label{as:comppr}
\begin{split}
&\textup{Let $u^1_\eps$ and $u^2_\eps$ be a subsolution and a supersolution of \eqref{eq:expDPP1} with boundary}\\ &\textup{data $G_1$ and $G_2$ respectively. If $G^1\leq G^2$ on $O$ then $u_\eps^1\leq u_\eps^2$ in $\Omega_E$.} 
\end{split}
\end{equation}

\begin{remark}
\begin{enumerate}[(a)]
\item   the assumption \eqref{as:comppr} is  the  comparison principle for \eqref{eq:expDPP1}. 
\medskip
\item $u_\eps$ being a subsolution (resp. supersolution) of the dynamic programming principle \eqref{eq:expDPP1} means that $u_\eps\leq A_\eps [u_\eps]$ in $\Omega$ and $u_\eps\leq G$ in $O$ (resp. $u_\eps\geq A_\eps [u_\eps]$ in $\Omega$ and $u_\eps\geq G$ in $O$).
\item \eqref{as:stab}, \eqref{as:stab2} and \eqref{as:comppr} are the natural assumptions and are satisfied by the dynamic programming principles in the literature (see Section \S \ref{sec:examples}).  For a general presentation on this subject we refer to  \cite{LPS14, LA15}.

\end{enumerate}
\end{remark}

\subsection{Main results}
Here is our  general convergence result regarding $C^2$-domains.
\begin{theorem}\label{thm:main}
Assume $p\in(1,\infty]$, $\Omega\subset \R^d$ be a bounded $C^2$-domain and $g\in C(\partial\Omega)$. Let $\{A_\eps\}_{\eps>0}$ be a mean value property for the $p$-Laplacian.  Let also $\{u_\eps\}_{\eps>0}$ be a sequence of solutions of the corresponding dynamic programming principle \eqref{eq:expDPP1}  satisfying  the assumption \eqref{as:stab}. Then we have that 
\[
\textup{$u_\eps\to v$  uniformly in $\overline{\Omega}$ as $\eps\to0$,  where $v$ is the unique viscosity solution of \eqref{dproblem}.}
\]
\end{theorem}

Once convergence for smooth domains is established, we will be able to generalize it to general  Lipschitz domains  using barrier arguments by requiring a little bit more of our the dynamic programming principle \eqref{eq:expDPP1}. We have the following result.
\begin{theorem}\label{thm:main2}
Let $p\in(1,\infty]$, $\Omega\subset \R^d$ be a bounded Lipschitz domain and $g\in C(\partial\Omega)$. Let $\{A_\eps\}_{\eps>0}$ be a mean value property for the $p$-Laplacian.  Let also $\{u_\eps\}_{\eps>0}$ be a sequence of solutions of the corresponding \eqref{eq:expDPP1}  satisfying  the assumptions \eqref{as:stab2} and \eqref{as:comppr}. Then we have that 
\[
\textup{$u_\eps\to v$  uniformly in $\overline{\Omega}$ as $\eps\to0$,  where $v$ is the unique viscosity solution of \eqref{dproblem}.}
\]
\end{theorem}
The proofs of Theorem \ref{thm:main} and Theorem \ref{thm:main2} will require several steps that we will  present in sections
\S \ref{sec:SUPC2domains} and \S \ref{sec:conv}. 

We  emphasize that the results of Theorem \ref{thm:main} and Theorem \ref{thm:main2} apply to more general dynamic programming principles than \eqref{eq:expDPP1}. The motivation of studying such generalisations is that, unlike \eqref{eq:expDPP1}, they are known to produce continuous solutions (cf. \cite{Har16}). For the sake of clarity,   we delay the statements of these results to section \S \ref{sec:contDPP} (see Theorem \ref{thm:main3} and Theorem \ref{thm:main4}).

\section{Strong uniqueness property for the Dirichlet problem in $C^2$-domains}\label{sec:SUPC2domains}
Our main strategy to prove convergence of solutions of  the dynamic programming principle \eqref{eq:expDPP1} to the solution of the Dirichlet problem \eqref{dproblem} is suggested by an argument introduced in \cite{BS91}, where the so-called \textit{strong uniqueness property} plays a crucial role.

We introduce a generalized version of viscosity solutions that takes into account the boundary condition in a weaker sense as done in \cite{BP87, BP88} and \cite{I89}. As we will see later, these solutions naturally appear as  limits of solutions of approximation problems akin to  dynamic programming principles. 

\begin{definition}\label{def:genVS}
Under the same hypothesis as in Definition \ref{def:ViscSub} (resp. Definition \ref{def:ViscSuper}), we say that $u$ is a  \textbf{generalized viscosity $p$-subsolution} (resp. \textbf{$p$-supersolution})  of the Dirichlet problem \eqref{dproblem} if the interior condition
\eqref{eq:intC1} (resp. \eqref{eq:intC2}) holds  and the condition at the boundary \eqref{eq:BC1} (resp. \eqref{eq:BC2}) is replaced by
\begin{eqnarray}
 \min\{-\Delta_p\phi(x_0), u(x_0)-g(x_0)\}  & \leq & 0,\quad  x_{0} \in\partial\Omega  \\  
(\textup{resp.  }    \max\{-\Delta_p\phi(x_0), u(x_0)-g(x_0)\} &\geq & 0, \quad  x_{0} \in\partial\Omega  ).
 \end{eqnarray} 

\end{definition}

The strong uniqueness property states the following:

\begin{propo}\label{prop:SUP}(Strong Uniqueness Property) Assume $p\in(1,\infty]$, $\Omega\subset \R^d$ be a $C^2$-domain, and $g\in C(\partial \Omega)$. 
Let $u$ and $v$ be a generalized viscosity $p$-subsolution and a $p$-supersolution respectively. Then $u\leq v$ in $\overline{\Omega}$.
\end{propo}

To show Proposition \ref{prop:SUP} we will prove that, in fact,  the notions of viscosity $p$-subsolution and generalized viscosity $p$-subsolution are equivalent for $C^2$-domains and then use the standard comparison principle for viscosity solutions.

\begin{theorem}\label{thm:equivSol}
Assume $p\in(1,\infty]$, $\Omega\subset \R^d$ be a $C^2$-domain and $g\in C(\partial \Omega)$. The following are equivalent
\begin{enumerate}[(a)]
\item\label{thm:equivSol:item-a} $u$ is a viscosity $p$-subsolution (resp. $p$-supersolution) of \eqref{dproblem}.
\item\label{thm:equivSol:item-b} $u$ is a generalized viscosity $p$-subsolution (resp. $p$-supersolution) of \eqref{dproblem}.
\end{enumerate}
\end{theorem}
The proof of Theorem \ref{thm:equivSol} will be given later in this section. As mentioned before, Proposition \ref{prop:SUP} follows trivially from Theorem \ref{thm:equivSol} and the following standard comparison principle for viscosity solutions (cf. \cite{JLM01}, \cite{LM17}):

\begin{theorem}\label{thm:clasicalcompprinc}(Comparison principle for viscosity solutions) Let $p\in(1,\infty]$, $\Omega\subset \R^d$ be a Lipschitz domain, and $g\in C(\partial \Omega)$.
Let $u$ and $v$ be viscosity $p$-subsolution and $p$-supersolution respectively of \eqref{dproblem}. Then we have $u\leq v$ in $\overline{\Omega}$.
\end{theorem}

\begin{proof}[Proof of Theorem \ref{thm:equivSol}]
We will prove the result for subsolutions since for supersolutions follows exactly in the same way.
It is trivial to show that \eqref{thm:equivSol:item-a}$\implies$\eqref{thm:equivSol:item-b} since 
\[
\min\{-\Delta_p\phi(x_0), u(x_0)-g(x_0)\} \leq u(x_0)-g(x_0) \leq 0.
\]
To prove \eqref{thm:equivSol:item-b}$\implies$\eqref{thm:equivSol:item-a} we need some more work. The proof is inspired in \cite{BB95}, where the case $p=2$ is presented. The idea of the proof is to try to find a suitable test function $\phi$ at $x_0\in \partial \Omega$ such that $-\Delta_p\phi(x_0)>0$. This will immediately imply that 
\[
\min\{-\Delta_p\phi(x_0), u(x_0)-g(x_0)\}  \leq 0 \iff u(x_0)-g(x_0)\leq0
\]
and will conclude the proof.
\end{proof}
\emph{1.  Regularity of the domain and the distance function.}

Consider the function $d(x):=\textrm{dist}(x, \partial\Omega)$. Since $\Omega$ is a $C^2$-domain,  for each $x_{0}\in \partial\Omega$ we can find an open set $U$ such that $x_0\in U$ and $d\in C^2_b(\overline{U})$. Moreover, we have that $d(x)>0$ if $x\in \overline{U}\cap \Omega$, $d(x)=0$ if $x\in \overline{U}\cap\partial\Omega$ and 
\begin{equation*}\label{eq:propdist}
|\nabla d(x)|=|-\vec{\mathbf{n}}(x)|=1 \quad \textup{and} \quad \|D^2d(x)\|_{L^\infty(\overline{U})}<+\infty.
\end{equation*}
We want to emphasize that the point of making the condition local in $U$ is to avoid the ridge points of the domain, where the distance function is not smooth. 

\emph{2.  Properties of a suitable test function (I), after \cite{BB95}}.
We  include the details of this step for the convenience of the reader.
Fix a point $x_{0}\in\partial\Omega$. 
For small $\eps>0$  consider the function
$$\phi_{\eps}(y)=\frac{|y-x_{0}|^{4}}{\eps^{4}}+ \frac{d(y)}{\eps^{2}}-\frac{d(y)^{2}}{2 \eps^{3}},$$
which is smooth in $\overline{U}$. 
Note that $u(y)-\phi_{\eps}(y)\in\USC(\overline{U}\cap\overline{\Omega})$.  In particular, if we consider the set
$$ K_{\eps}:= \{ y\colon d(y)\le \eps\} \cap \overline{U}\cap\overline{\Omega}.$$
then we also have that $u(y)-\phi_{\eps}(y)\in\USC(K_\eps)$. Therefore
$u-\phi_{\eps}$ attains its maximum in $K_{\eps}$ at some point $y_{\eps}\in K_{\eps}$.
Observe that 
\begin{equation}\label{eq:24}
\frac{d(y)}{\eps^{2}}-\frac{d(y)^{2}}{2 \eps^{3}}= \frac{d(y)}{\eps^2}\left(1-\frac{d(y)}{2\eps}\right)\geq0 \quad \textup{for all} \quad y\in K_\eps.
\end{equation}
Since $\phi_\eps(x_0)=0$, we have that
\begin{equation}\label{eq:25}
\begin{split}
u(x_{0})  &=   u(x_{0})-\phi_{\eps}(x_{0})\leq u(y_{\eps}) -\phi_{\eps}(y_{\eps})\\
&=  u(y_{\eps}) -\frac{|y_{\eps}-x_{0}|^{4}}{\eps^{4}}- \frac{d(y_{\eps})}{\eps^{2}}+\frac{d(y_{\eps})^{2}}{2 \eps^{3}}.
\end{split}
\end{equation}
Using \eqref{eq:24} in \eqref{eq:25} we get
\[
u(x_{0}) \leq u(y_{\eps}) -\frac{|y_{\eps}-x_{0}|^{4}}{\eps^{4}}.
\]

If $u(x_{0})=-\infty$ we have nothing to prove. Therefore, we assume that $u(x_{0})> - \infty$.
Since $u$ is bounded from above because it is USC in $\overline{\Omega}$,  the previous estimate ensures that $\frac{|y_{\eps}-x_{0}|}{\eps}$ is also bounded. It also implies that 
$\lim_{\eps\to0}y_{\eps}=x_{0}$
and 
\begin{equation}\label{eq:orderUclosetoMax}
u(x_{0})\le u(y_{\eps}).
\end{equation}
By the upper semicontinuity of $u$ we get
$$u(x_{0})\le \liminf_{\eps\to0}u(y_{\eps})\le  \limsup_{\eps\to0}u(y_{\eps})\le u(x_{0}).$$
Next, we take $\displaystyle\liminf_{\eps\to0}$ in (\ref{eq:25}) to get 
$$
 u(x_{0})  \le u(x_{0}) + \liminf_{\eps\to0}
 \left[
  -\frac{|y_{\eps}-x_{0}|^{4}}{\eps^{4}}- \frac{d(y_{\eps})}{\eps^{2}}+\frac{d(y_{\eps})^{2}}{2 \eps^{3}}
 \right],
$$
so that
$$\limsup_{\eps\to0}\left[
  \frac{|y_{\eps}-x_{0}|^{4}}{\eps^{4}}+\left( \frac{d(y_{\eps})}{\eps^{2}}-\frac{d(y_{\eps})^{2}}{2 \eps^{3}}\right)\right]
 \le 0.
$$
Since both terms in the left hand side are non-negative, we conclude that they both go to zero. Since we also have
$$\frac{d(y_{\eps})}{2 \eps^{2}} \le \frac{d(y_{\eps})}{\eps^{2}}-\frac{d(y_{\eps})^{2}}{2 \eps^{3}},$$
we conclude that
$$ d(y_{\eps})=o(\eps^{2}).$$
In particular we cannot have $d(y_{\eps})=\eps$ for small $\eps$. Since $y_{\eps}\to x_{0}$ then  $y_\eps\not\in \partial U$. We conclude that $y_{\eps}$ is a  point of  \textit{local} maximum for $u-\phi_{\eps}$ in
$K_{\eps}$, thus we can extend $\phi_{\eps}$ so that $y_{\eps}$ is point of local maximum for $u-\phi_{\eps}$  in $\overline{\Omega}$.

\emph{3.  Properties of a suitable test function (II). } \nopagebreak

In order to prove that $\phi_\eps$ is a suitable test function at $y_\eps$ for Definition \ref{def:genVS} we show that $\nabla\phi(y_\eps)\not=0$. Differentiating we get
\begin{equation}\label{eq:nablaphieps}
\nabla\phi_{\eps}(y)=\frac{4 |y-x_{0}|^{2}}{\eps^{4}} (y-x_0)+ \frac{\nabla d(y)}{\eps^{2}}-\frac{d(y)}{\eps^{3}}\nabla d(y),
\end{equation}
and
\begin{equation*}
\begin{split}
|\nabla\phi_{\eps}(y)|^2=& \displaystyle{\frac{16 |y-x_{0}|^{6}}{\eps^{8}} +\frac{1}{\eps^4}+\frac{d^2(y)}{\eps^6}
+ \frac{8 |y-x_0|^2}{\eps^{6}} \langle y-x_0,\nabla d(y)\rangle}\\
&- \displaystyle{\frac{8 d(y)|y-x_0|^2}{\eps^{7}}\langle y-x_0,\nabla d(y)\rangle-\frac{2}{\eps^5}d(y)}.
\end{split}
\end{equation*}
We recall that, from steps 1 and 2, we have that  $|y_\eps-x_0|\le c \eps$, $\Delta d$ is bounded, $|\nabla d|=1$, and $d(y_\eps)=o({\eps}^2)$.
We conclude that 
\begin{equation*}\label{gradientbound}
\frac{1}{\tilde{c} \eps^4}\le  |\nabla\phi_{\eps}(y_\eps)|^2\le \frac{\tilde{c}}{ \eps^4},
\end{equation*} for some constant $\tilde{c}>0$ and $\eps$ small enough. In particular, we always have $ \nabla\phi_{\eps}(y_\eps) \not= 0$.

\emph{4. Positivity of $-\Delta_p\phi_\eps(y_\eps)$.}

We will estimate $-\Delta\phi_\eps(y_\eps)$ and $-\Delta_\infty^N\phi_\eps(y_\eps)$ separately and the result for $-\Delta_p\phi_\eps(y_\eps)$ will follow from the interpolation between these operators given in Remark \ref{rem:interpPlap}.

Differentiating \eqref{eq:nablaphieps} and using one more time that $|y_\eps-x_0|\le c \eps$, $\Delta d$ is bounded and $|\nabla d|=1$  we get
\begin{equation*}
\begin{split}
\Delta \phi_{\eps}(y_{\eps}) &=  \frac{(4n+8)|y_{\eps}-x_{0}|^{2}}{\eps^{4}}+ \frac{\Delta d(y_{\eps})}{\eps^{2}}-\frac{\left[2 |\nabla d(y_{\eps})|^{2}+ 2 d(y_\eps) \Delta d(y_{\eps})
\right]}{2\eps^{3}}  \\
& \le  c_{1}\frac{4n+8}{\eps^{2}} + \frac{c_{2}}{\eps^{2}}-\frac{1}{\eps^{3}}+ \frac{d(y_{\eps})
|\Delta d (y_{\eps})|}{\eps^{3}},
\end{split}
\end{equation*}
The previous estimate, together with $d(y_{\eps})=o(\eps^{2})$ shows that for some $C>0$ we have that 
\begin{equation}\label{lapbound}
\Delta \phi_{\eps}(y_{\eps})\le \frac{C}{\eps^{2}} - \frac{1}{\eps^{3}}
\end{equation}
We will find an estimate for  $\Delta_\infty^N\phi_\eps(y_\eps)$  using
 $$ |\nabla\phi_{\eps}(y)|^2 \Delta_\infty^N \phi_\eps (y)=\Delta_\infty \phi_\eps(y)= \langle D^2\phi_\eps (y)\cdot
\nabla\phi_{\eps}(y),\nabla\phi_{\eps}(y)
\rangle.
$$ 
Note that
$$D^2\phi_\eps (y)= \frac{4|y-x_0|^2 I_{n} }{\eps^4} + \frac{8(y-x_0)\otimes (y-x_0)}{\eps^4} 
+\frac{D^2d(y)}{\eps^2}-\frac{d(y)D^2 d(y)}{\eps^3}- \frac{\nabla d(y) \otimes \nabla d(y)}{\eps^3}
$$ so that the expression $\langle D^2\phi_\eps (y)\cdot
\nabla\phi_{\eps}(y),\nabla\phi_{\eps}(y)
\rangle 
$ has 45 terms. One can compute them explicitly, however it is interesting first to notice the order of magnitude of each term when $y=y_\eps$. More precisely, using once again that $|y_\eps-x_0|\le c \eps$, $D^2 d$ is bounded, $|\nabla d|=1$ and $d(y_{\eps})=o(\eps^{2})$ we have that 
\begin{equation*}
\begin{split}
D^2\phi_\eps (y_\eps)=&\underbrace{\frac{4|y_\eps-x_0|^2 I_{n} }{\eps^4}}_{\sim \eps^{-2}} + \underbrace{\frac{8(y_\eps-x_0)\otimes (y_\eps-x_0)}{\eps^4}}_{\sim \eps^{-2}}
+\underbrace{\frac{D^2d(y_\eps)}{\eps^2}}_{\sim \eps^{-2}}\\
&-\underbrace{\frac{d(y_\eps)D^2 d(y_\eps)}{\eps^3}}_{\sim \eps^{-1}}- \underbrace{\frac{\nabla d(y_\eps) \otimes \nabla d(y_\eps)}{\eps^3}}_{\sim\eps^{-3}}
\end{split}
\end{equation*}
and
\[
\nabla\phi_{\eps}(y_\eps)=\underbrace{\frac{4 |y_\eps-x_{0}|^{2}}{\eps^{4}} (y_\eps-x_0)}_{\sim \eps^{-1}}+ \underbrace{\frac{\nabla d(y_\eps)}{\eps^{2}}}_{\sim \eps^{-2}}-\underbrace{\frac{d(y_\eps)}{\eps^{3}}\nabla d(y_\eps)}_{\sim \eps^{-1}},
\]
where by $\sim$ we mean that the absolute size is at the most of the denoted magnitude.
It is clear that the positive dominating  term in  $\Delta_\infty \phi_\eps(y_\eps)$ is bounded by $C/\eps^{6}$ for $\eps$ small enough. One the other hand, the negative dominating term is precisely given by 
\[
-\left\langle \frac{\nabla d(y_\eps) \otimes \nabla d(y_\eps)}{\eps^3}\cdot \frac{\nabla d(y_\eps)}{\eps^{2}}, \frac{\nabla d(y_\eps)}{\eps^{2}} \right\rangle=- \frac{1}{\eps^7}.
\]
We conclude then that $\Delta_\infty \phi_\eps(y)\leq C\eps^{-6}- \eps^{-7}$, and 
\begin{equation}\label{infylapbound}
\Delta_\infty^N \phi_{\eps}(y_{\eps})\le \frac{C_1}{\eps^{2}} - \frac{C_2}{\eps^{3}}.
\end{equation}
 Combining \eqref{lapbound} and  \eqref{infylapbound} we get that 
\begin{equation*}\label{plapbound}
\begin{split}\frac{1}{p}
|\nabla \phi_{\eps}(y_\eps)|^{{2-p}}\Delta_{p}\phi_{\eps}(y_\eps)&=\Delta_{p}^N\phi_{\eps}(y_\eps)= {\frac{1}{p} \Delta \phi_\eps(y_\eps) + \frac{(p-2)}{p} \Delta_\infty^N \phi_\eps (y_\eps)}\\
&\leq \frac{p-1}{p}\left(\frac{C_1}{\eps^2}-\frac{C_2}{\eps^3}\right).
\end{split}
\end{equation*}
Since by assumption $p>1$ and $|\nabla \phi_{\eps}(y_\eps)|\not=0$ we get that for $\eps$ small enough we have
\begin{equation}\label{eq:contraest}
-\Delta_p\phi_\eps(y_\eps)>0.
\end{equation}
The same conclusion follows for $p=\infty$ directly from \eqref{infylapbound}.

\emph{5. Conclusion.}\par
 From  the definition of generalized viscosity solution we get 
\[
\min\{
-\Delta_{p} \phi_{\eps}(y_{\eps}),\,  u(y_\eps)-g(y_{\eps})
\}\le 0.
\]

From \eqref{eq:contraest}, it follows that $u(y_\eps) \le g(y_\eps)$. By taking limits, using \eqref{eq:orderUclosetoMax}, and the continuity of $g$ the result follows. 
\section{Proof of convergence: A numerical analysis approach}\label{sec:conv}
In order to prove convergence of solutions of the dynamic programming principle \eqref{eq:expDPP1} to the solution of  the Dirichlet problem \eqref{dproblem} we need to ensure three properties of \eqref{eq:expDPP1}: Stability, Monotonicity and Consistency. 

These names are taken from the numerical analysis framework. Stability is precisely given by the assumption \eqref{as:stab}. Monotonicity of \eqref{eq:expDPP1} is a standard property that follows directly from the properties of an average operator. On the other hand, the consistency needed below in this context is slightly different than the usual one, but it will be a consequence of the concept of mean value property.

Let us define the following operator for a  smooth functions $\phi$ that depends explicitly on $\eps>0$, $x\in \Omega_E$, $\phi(x)$ and $\phi$ (through the average $A_{\eps}$):
\begin{equation}\label{eq:DPPscheme}
S(\eps, x, \phi(x), \phi)=\left\{\begin{array}{cccl}
\frac{1}{c_{p,d} \eps^2}\left(\phi(x)-A_\eps[\phi](x)\right)& \text{ if } &x\in\Omega\\
\phi(x)-G(x) & \text{ if } &x\in {O}
\end{array}\right.
\end{equation}

Then, the dynamic programming principle \eqref{eq:expDPP1} can be formulated as the following scheme:
\begin{equation}\label{eq:DPP2}
S(\eps, x, u_\eps(x), u_\eps)=0 \quad \textup{for all} \quad x\in \Omega_E.
\end{equation}

\begin{remark}
The scheme \eqref{eq:DPP2} encodes both the interior equation and the boundary condition.
\end{remark}

\subsection{Monotonicity}
We have the following result of monotonicity for the scheme \eqref{eq:DPP2}.

\begin{lemma}\label{lem:mono}
Let $\Omega \subset \R^d$ be a bounded domain and $A_\eps$ be an average. Then $S$ defined by \eqref{eq:DPPscheme} is monotone, that is,
for all $\eps>0$, $x\in \Omega_E$,  $t\in \R$ and $u,v\in B(\Omega_E)$ such that $u\leq v$ we have that 
\[
S(\eps, x, t, v)\leq S(\eps, x, t, u)
\]
\end{lemma}
\begin{proof}
Since $A_\eps$ is an average, we have that $A_\eps[u]\leq A_\eps[v]$, so if $x\in \Omega$, then
\[
S(\eps, x, t, v)=\frac{1}{c_{p,d} \eps^2}\left(t-A_\eps[v](x)\right)\leq\frac{1}{c_{p,d} \eps^2}\left(t-A_\eps[u](x)\right)= S(\eps, x, t, u).
\]
On the other hand, if $x\in O$, 
\[
S(\eps, x, t, v)=t-G(x)=S(\eps, x, t, u).\qedhere
\]
\end{proof}
\subsection{Consistency} The following consistency result incorporates  both the interior equation and the boundary condition of \eqref{eq:expDPP1}. Note that it is related to the concept of generalized viscosity solution.

\begin{lemma}\label{lem:cons}
Let $\{A_\eps\}_{\eps>0}$ be a  mean value property for the $p$-Laplacian and let $S$ be given by \eqref{eq:DPPscheme}. Then, for all $x\in \overline{\Omega}$ and $\phi\in C_b^\infty(\Omega_{E})$ such that $\nabla\phi(x)\not=0$, we have that
\[
\limsup_{\eps\to0,\ y\to x,\ \xi \to 0} S(\eps, y, \phi(y)+\xi, \phi+\xi) = \left\{\begin{array}{cccl}
-\Delta_p^N\phi(x)& \text{ if } &x\in\Omega\\
\max\{-\Delta_p^N\phi(x), \phi(x)-G(x)\}& \text{ if } &x\in \partial{\Omega}
\end{array}\right.
\]
and 
\[
\liminf_{\eps\to0,\ y\to x,\ \xi \to 0} S(\eps, y, \phi(y)+\xi, \phi+\xi) = \left\{\begin{array}{cccl}
-\Delta_p^N\phi(x)& \text{ if } &x\in\Omega\\
\min\{-\Delta_p^N\phi(x), \phi(x)-G(x)\}& \text{ if } &x\in \partial{\Omega}.
\end{array}\right.
\]
\end{lemma} 
\begin{proof}
We present the proof for the $\limsup$, since the other one is similar. First note that, since $A_\eps[\phi+\xi]=A_\eps[\phi]+\xi$, we have that
\begin{equation*}
S(\eps, y, \phi(y)+\xi, \phi+\xi)=\left\{\begin{array}{cccl}
\frac{1}{c_{p,d} \eps^2}\left(\phi(y)-A_\eps[\phi](y)\right)& \text{ if } &y\in\Omega\\
\phi(y)-G(y)+\xi & \text{ if } &y\in O.
\end{array}\right.
\end{equation*}

Fix first $x\in \Omega$ and let $B$ denote a ball centred at $x$ and radius smaller than $d(x,\partial \Omega)/2$ (so that    $B\subset\Omega$) and such that $\nabla \phi(y)\not=0$ for all $y\in B$ (we can assume this by regularity of $\phi$ and the fact that $\nabla\phi(x)\not=0$). Then,  by the fact that $y\in\Omega$ (limit as $\xi\to0$),  the uniformity in the mean value property in \eqref{eq:MVP} (limit as $\eps\to0$) and  by regularity of $-\Delta_p^N \phi$ (limit as $y\to x$), we have
\begin{equation*}
\begin{split}
\limsup_{\eps\to0,\ y\to x,\ \xi \to 0} S(\eps, y, \phi(y)+\xi, \phi+\xi)&= \limsup_{\eps\to0,B\ni y\to x,\ \xi \to 0} S(\eps, y, \phi(y)+\xi, \phi+\xi)\\
&=\limsup_{\eps\to0,B\ni y\to x} \frac{1}{c_{p,d} \eps^2}\left(\phi(y)-A_\eps[\phi](y)\right)\\
& = \limsup_{\eps\to0, B\ni y\to x} (-\Delta_p^N \phi(y)+ o_\eps(1))= -\Delta_p^N \phi(x).
\end{split}
\end{equation*}
Next, let $x\in \partial \Omega$. We can approach $x$ both from points $y\in \Omega$ and $y\in O$. Again, we can assume that $\nabla \phi(y)\not=0$ for all $y$ close enough to $x$. We then have
\begin{equation*}
\begin{split}
&\limsup_{\eps\to0,\ y\to x,\ \xi \to 0} S(\eps, y, \phi(y)+\xi, \phi+\xi)\\
&\hspace{2 em}= \max \left\{ \limsup_{\eps\to0,\Omega\ni y\to x} \frac{1}{c_{p,d} \eps^2}\left(\phi(y)-A_\eps[\phi](y)\right), \limsup_{O\ni y\to x, \xi \to 0} (\phi(y)-G(y)+\xi)\right\}\\
&\hspace{2 em}= \max \left\{-\Delta_p^N \phi(x), \phi(x)-G(x)\right\}.\qedhere
\end{split}
\end{equation*}
\end{proof}
\subsection{Proof of convergence for $C^2$-domains}
We are now ready to prove our most general convergence result regarding $C^2$-domains.
\begin{proof}[Proof of Theorem \ref{thm:main}]
This argument is  from \cite{BS91} in the context of numerical schemes. Define
\begin{equation}\label{barles}
\vv(x)=\limsup_{\eps\to0, y\to x  } u_\eps(y), \qquad \underline{v}(x)=\liminf_{\eps\to0, y\to x  } u_\eps(y).
\end{equation}
By  the assumption \eqref{as:stab} both $\vv$ and $\underline{v}$ are bounded functions in $\overline{Q}_T$. Also, by definition, $\vv$ is USC and $\underline{v}$ is LSC and $\underline{v}\leq \vv$.

Assume for a moment (we will prove it later) that $\vv$ is a generalized viscosity $p$-subsolution and $\underline{v}$ is a generalized viscosity $p$-supersolution of \eqref{dproblem}. Then, the strong uniqueness property given by Proposition \ref{prop:SUP}, ensures also that $\underline{v}\geq \vv$. In fact, this proves that $v:=\underline{v}=\vv$ is a generalized viscosity $p$-solution of \eqref{dproblem} and also that $u_\eps\to v$  as $\eps\to0$ uniformly in $\overline{\Omega}$. The equivalence of notations of viscosity solutions given by Theorem \ref{thm:equivSol} concludes the proof.

Next, we prove that $\vv$ is a generalized viscosity $p$-subsolution of \eqref{dproblem}. Let $\phi\in C^\infty_b(\Omega_E)$ and $x_0\in \overline{\Omega}$ such that $\vv(x_0)=\phi(x_0)$, $\vv< \phi$ if $x\not=x_0$ ($\vv-\phi$ reaches a global maximum on $\overline{\Omega}$) and $\nabla\phi(x_0)\not=0$. Then for all $x\in \overline{\Omega}$, we have that
\[
\vv(x)-\phi(x)\leq0=\vv(x_0)-\phi(x_0).
\]
Thus, we can find a sequence $\{\eps_n\}_{n\geq0}$ and $\{y_n\}_{n\geq0} \subset \Omega_E$, such that
\[
\eps_n\to 0^+, \quad y_n\to x_0, \quad  u_{\eps_n}(y_n)\to \vv (x_0)  \quad \textup{and}\quad \nabla \phi(y_n) \not=0
\]
with $y_n$ being a global max of $u_{\eps_n}-\phi$.  Let now $\xi_n:=u_{\eps_n}(y_n)-\phi(y_n)$. We have that $\xi_n \to0$ and $u_{\eps_n}(x)-\phi(x)\leq \xi_n$, that is,
\[
u_{\eps_n}(x)\leq \phi(x) +\xi_n \quad \textup{in}\quad \Omega_E.
\]
Since $A_\eps$ is an average, the monotonicity given by  Lemma \ref{lem:mono} ensures that
\begin{equation*}
\begin{split}
0&=S(\eps_n,y_n, u_{\eps_n}(y_n),u_{\eps_{n}})\\
&=S(\eps_n,y_n, \phi(y_n)+\xi_n,u_{\eps_{n}})\\
&\geq S(\eps_n,y_n, \phi(y_n)+\xi_n,\phi+\xi_n).\\
\end{split}
\end{equation*}
As usual, we can assume that $\nabla \phi(y)\not=0$ for all $y$ close enough to $x_0$. Consistency given by Lemma \ref{lem:cons} then shows
\begin{equation*}
\begin{split}
0&\geq \liminf_{\eps_n\to0,\ y_n\to x_0, \ \xi_n\to0} S(\eps_n,y_n, \phi(y_n)+\xi_n,\phi+\xi_n)\\
&\geq \liminf_{\eps\to0,\ y\to x_0, \ \xi\to0} S(\eps,y, \phi(y)+\xi,\phi+\xi)\\
&\geq \left\{\begin{array}{cccl}
-\Delta_p^N\phi(x_0)& \text{ if } &x_0\in\Omega\\
\min\{-\Delta_p^N\phi(x_0), \phi(x_0)-G(x_0)\}& \text{ if } &x_0\in \partial{\Omega}.
\end{array}\right.
\end{split}
\end{equation*}
Since $\phi(x_0)=\vv(x_0)$ and $G(x_0)=g(x_0)$ if $x_0\in \partial \Omega$, we have concluded that 
\begin{align*}
 -\Delta_p^N\phi(x_0)&\leq0 \quad \text{ if } \quad  x_0\in \Omega,\\
\min\{-\Delta_p^N\phi(x_0), \vv(x_0)-g(x_0)\}&\leq0 \quad \text{ if } \quad  x_0\in\partial\Omega.
\end{align*}
Moreover, since $\nabla \phi(x_0)\not=0$ we have that $-\Delta_p^N\phi(x_0)\leq0$ if and only if $-\Delta_p\phi(x_0)\leq0$ and thus $\vv$ is a generalized viscosity $p$-subsolution of \eqref{dproblem}, and the proof is completed. 
\end{proof}

\subsection{Proof of convergence for Lipschitz domains} Next we prove Theorem \ref{thm:main2}, i.e.\ we present the proof of convergence for Lipschitz domains following the proof sketched in \cite{MPR12} in the case of $p$-harmonious functions. See Figure \ref{fig:convLipDom} below for a graphical explanation of the proof.

Since $\Omega$ is Lipschitz, it is clear that $\Omega$ satisfies the following regularity condition which is the one that we actually use in the proof:
\begin{equation*}
\begin{split}
&\text{There exists} \,  \bar{\delta}>0 \, \text{and} \,  \mu\in(0,1) \,  \text{such that for every} \,  \delta\in(0,\bar{\delta}) \,  \text{and} \,  y\in\partial\Omega \, \\
&\text{there exists a ball}  \, B_{\mu\delta}(z) \,  \text{strictly contained in} \, B_\delta(y)\setminus\Omega.
\end{split}
\end{equation*}
Set $\Omega_{\eps}=\Omega\cup\Gamma_\eps$.
Let $u_\eps$ be as in Theorem \ref{thm:main2}.
Fix $\delta\in(0,\bar{\delta})$.
For $y\in\partial\Omega$ consider: 
\begin{equation}\label{Beps}
m^\eps(y) :=\sup_{B_{5\delta}(y)\cap \Gamma_\eps}G  \quad
\text{and}\quad 
M^\eps :=\sup_{ \Gamma_\eps}G.
\end{equation}
Assume momentarily that $p\not= d$. 
Fix a  number $\theta\in (0,1)$ depending only on $\mu$, $d$ and $p$,  to be determined later. 
For $k\geq 0$ define $\delta_{k}=\delta/4^{k-1}$
and 
\begin{equation}\label{Meps}
M_k^\eps(y):=m^\eps(y)+\theta^k(M^\eps-m^\eps(y)).
\end{equation}
By the regularity assumption on $\Omega$, there exist balls $B_{\mu \delta_{k+1}}(z_k)$ contained in $B_{\delta_{k+1}}(y)\setminus \Omega$ for all $k\in\mathbb{N}$. Note that $\mu$ is independent of $k$ and $\delta$.\par
Next, we present the basic iteration lemma which follows. 
We construct smooth  barriers based on the fundamental solution on appropriate shrinking rings. We use the convergence result on smooth domains to get estimates in the Lipschitz case.
 The proof, see (\ref{eq:choice-theta}), will show that the right choice  for $\theta$ is
\begin{equation}
\label{eq:theta}
\theta = \frac{1-\frac{1}{2}\left(\frac{\mu}{2-\mu}\right)^{\xi}-\frac{1}{2}\left(\frac{\mu}{2}\right)^\xi}{1-\left(\frac{\mu}{2}\right)^\xi}\in(0,1), 
\end{equation}
where $\xi=\frac{d-p}{p-1}$.
\begin{lemma}\label{Lemma induction k}
Fix $\eta>0$ and let $y\in\partial\Omega$ and $\eps_k>0$. Under the above notations, suppose that for all $\eps<\eps_k$ we have:
\begin{equation*}
u_\eps \leq M_k^\eps(y) \quad \text{in}\quad B_{\delta_k}(y)\cap \Omega. 
\end{equation*}
Then, either $M_k^\eps(y)-m^\eps(y)\leq \frac{\eta}{4}$ or there exists $\eps_{k+1}=\eps_{k+1}(\eta, \mu,\delta,d,p,G)\in(0,\eps_k)$  such that:
\begin{equation*}
u_\eps \leq M_{k+1}^\eps(y) \quad \text{in}\quad B_{\delta_{k+1}}(y)\cap \Omega
\end{equation*}
for all $\eps\leq \eps_{k+1}$. 
\end{lemma}
\begin{proof}
The idea of the proof is to iterate over annular domains  and use a comparison principle against a $p$-harmonic function with suitable boundary values. In this way, we get an upper bound for $u_{\eps}$ inside the annular domain that is smaller than the upper bound at the outer boundary of the annular domain we started with. On the next round, we take this improved upper bound as a boundary value at the outer boundary of the next smaller annular domain and repeat the argument to gain again the same multiplying factor $\theta$, and continuing in this way prove the result. The key estimates showing that we gain the uniform $\theta$ at each round are (\ref{improved0})--(\ref{eq:choice-theta}) below.

{\bf 1.}
For notational convenience, denote  $m=m^\eps(y)$, $M=M^\eps$ and $M_k=M_k^\eps(y)$.
Consider the barrier: 
\begin{equation*}
U_{k}(x):=\frac{a_{k}}{|x-z_{k}|^{\xi}}+b_{k}\quad \textup{where}\quad 
a_{k}=\frac{m-M_k}{1-(\mu/4)^\xi}(\mu\delta_{k+1})^\xi \quad \text{and} \quad b_{k}=\frac{M_k-(\mu/4)^\xi m}{1-(\mu/4)^\xi}.
\end{equation*}
Note that when $\xi\neq 0$ we have that $U_k$ is increasing in $|x-z_{k}|$, has non vanishing  gradient, and solves the problem
\begin{equation*}
\begin{cases}
\Delta_{p} U_k = 0 &\quad \text{in}\quad B_{\delta_k}(z_k)\setminus \overline{B}_{\mu\delta_{k+1}}(z_k) \\
U_k = m &\quad \text{on}\quad \partial B_{\mu\delta_{k+1}}(z_k)\\
U_k = M_k &\quad \text{on}\quad \partial B_{\delta_{k}}(z_k).
\end{cases}
\end{equation*}
In the case $p=d$, we use $$U_k(x)=a_k\log(|x-z_k|)+b_k$$
with suitable coefficients $a_k$ and $b_k$. We give details only in the case $p\not=d$.\par
We will establish several upper bounds for $\eps_{k+1}$, and   take $\eps_{k+1}$ to be the minimum of such bounds. 
First, let $\eps_{k+1}=\frac{\mu\delta_{k+1}}{2}$.
For $\eps\leq \eps_{k+1}$, extend the barrier $U_k$ to the ring $$R_{k,\eps}= B_{\delta_k+2\eps}(z_k)\setminus \overline{B}_{\mu\delta_{k+1}-2\eps}(z_k).$$
Let $U^\eps_k$ be the solution of the the dynamic programming principle \eqref{eq:expDPP1} in the ring  $R_{k}=B_{\delta_k}(z_k)\setminus\overline{B}_{\mu\delta_{k+1}}(z_k)$ with boundary value $U_k$ on $R_{k,\eps}\setminus R_{k}$, the outer $\eps$-neighbourhood of $R_{k}$. 
Since $R_{k}$ is a smooth domain, by 
Theorem \ref{thm:main} we have that $U^\eps_k$ converges to $U_k$ uniformly in $R_{k, \epsilon}$ as $\eps\to 0$. 
Hence, given $$\gamma=\gamma(k,\mu, p, G,\eta)=\frac{1}{4}\frac{(\mu/(\mu-2))^\xi-(\mu/2)^\xi}{1-(\mu/4)^\xi}\frac{\eta}{4}>0,$$
there exists $\eps_{k+1}=\eps_{k+1}(\gamma)>0$ such that
\begin{align}
\label{eq:approx-rings}
|U^\eps_k-U_k|\leq \gamma
\end{align}
for $\eps\leq \eps_{k+1}$ and for every $x\in R_{k,\eps}$.
\medskip

{\bf 2.}
We define
\begin{equation*}
a= \frac{1-(\mu/2)^{\xi}}{1-(\mu/4)^{\xi}} \quad \text{and}\quad 
b =\frac{(\mu/2)^{\xi}-(\mu/4)^{\xi}}{1-(\mu/4)^{\xi}},
\end{equation*}
and note that $a+b=1$. 

Next, we prove the following claim
\begin{equation}\label{claim}
a u_\eps +b m\leq U_k+2\gamma \quad \text{in}\quad B_{\delta_k/2}(z_k)\cap\Omega,
\end{equation}
 for $\eps\leq \eps_{k+1}$.
We will use the comparison principle \eqref{as:comppr} in the $\eps$-neighbourhood of $B_{\delta_k/2}(z_k)\cap\Omega$ whose $\eps$-boundary is contained in  
$\Gamma_1^\eps\cup \Gamma_2^\eps$
where  $$\Gamma_1^\eps:=B_{\delta_k/2+\eps}(z_k)\cap\Gamma_\eps \quad \text{and} \quad \Gamma_2^\eps:=(B_{\delta_k/2+\eps}(z_k)\setminus \overline{B}_{\delta_k/2}(z_k))\cap\Omega,$$
see Figure \ref{fig:convLipDom}. \par
On $\Gamma_1^\eps$, we have $u_\eps=G\leq m$, and hence on $\Gamma_1^\eps$
$$a  u_\eps+b m\leq m=\inf_{R_k} U_k\leq U_k\leq U_k^\eps+\gamma,$$
where at the last step we used (\ref{eq:approx-rings}) and the fact $\Gamma_1^\eps\subset R_{k,\eps}$.\par
On $\Gamma_2^\eps$, we have $u_\eps\leq M_k$ by assumption, because $B_{\delta_k/2+\eps}(z_k)\subset B_{\delta_k}(y)$. 
For $x\in\partial B_{\delta_k/2}(z_k)$, we have $|x-z_{k}|=\delta_k/2$, and hence
\begin{equation}
\begin{split}
U_k(x)= a_k (\delta_k/2)^{-\xi}+b_k
=\frac{m-M_k}{1-(\mu/4)^{\xi}}(\mu/2)^{\xi}+\frac{M_k-(\mu/4)^{\xi}m}{1-(\mu/4)^{\xi}}
=a M_k+ b m,
\end{split}
\end{equation}
and by monotonicity of $U_k$ we get 
$U_k\geq a M_k+ b m$ in $\Gamma_2^\eps$.
Thus
\begin{equation*}\label{boundary2}
a u_\eps+ b m\leq a M_k+b m\leq U_k\leq U_k^\eps+\gamma
\end{equation*}
in $\Gamma_2^\eps$.
In conclusion, we have
$$
a u_\eps+b m\leq U_k^\eps+\gamma \quad\text{in}\quad \Gamma_1^\eps\cup \Gamma_2^\eps,
$$
and the claim in (\ref{claim}) follows by the condition \eqref{as:comppr} and by the fact that $\eps$-boundary of $B_{\delta_k/2}(z_k)\cap\Omega$ is contained in $\Gamma_1^\eps\cup \Gamma_2^\eps$.

\medskip

{\bf 3.}
Now consider the intersection  $B_{\delta_{k+1}}(y)\cap\Omega$. We have $B_{\delta_{k+1}}(y) \subset B_{(2-\mu)\delta_{k+1}}(z_k)$ and for $x\in B_{(2-\mu)\delta_{k+1}}(z_k)$ we have:
\begin{equation}\label{improved0}
\begin{split}
U_k(x)&\leq\frac{m-M_k}{1-(\mu/4)^{\xi}}(\mu\delta_{k+1})^{\xi}((2-\mu)\delta_{k+1})^{-\xi}+ \frac{M_k-(\mu/4)^{\xi}}{1-(\mu/4)^{\xi}}\\
&=b'm+a'M_k,
\end{split}
\end{equation}
where
$$ a'= \frac{1-(\mu/(2-\mu))^\xi}{1-(\mu/4)^\xi} \quad \text{and}\quad  b'=\frac{(\mu/(2-\mu))^\xi-(\mu/4)^\xi}{1-(\mu/4)^\xi}.$$
Also, note that 
$B_{\delta_{k+1}}(y)\subset B_{\delta_k/2}(z_k)$,
hence by \eqref{claim} we get
\begin{equation}\label{improved}
a u_\eps+b m\leq U_k+2\gamma \quad\text{in}\quad B_{\delta_{k+1}}(y)\cap\Omega.
\end{equation}  
Combining \eqref{improved0} and \eqref{improved}, for $x\in B_{\delta_{k+1}}(y)\cap\Omega$ and $\eps<\eps_{k+1}$, we get
\begin{equation}
\begin{split}
\label{eq:choice-theta}
u_\eps(x)&\leq 
\frac{b'- b}{a}m+\frac{a'}{a}M_k+\frac{2\gamma}{a}
\le  m+\frac{a'}{a}(M_k-m)+\frac{b' (M_k-m)}{2 a}\\
&=m+\theta(M_k-m)=m+\theta^{k+1}(M-m),
\end{split}
\end{equation}
because $\gamma\leq \frac{1}{4}\frac{(\mu/(\mu-2))^\xi-(\mu/2)^\xi}{1-(\mu/4)^\xi}(M_k-m)$ holds.
\end{proof}

\begin{figure}[h!]
\centering
\includegraphics[width=0.95\textwidth]{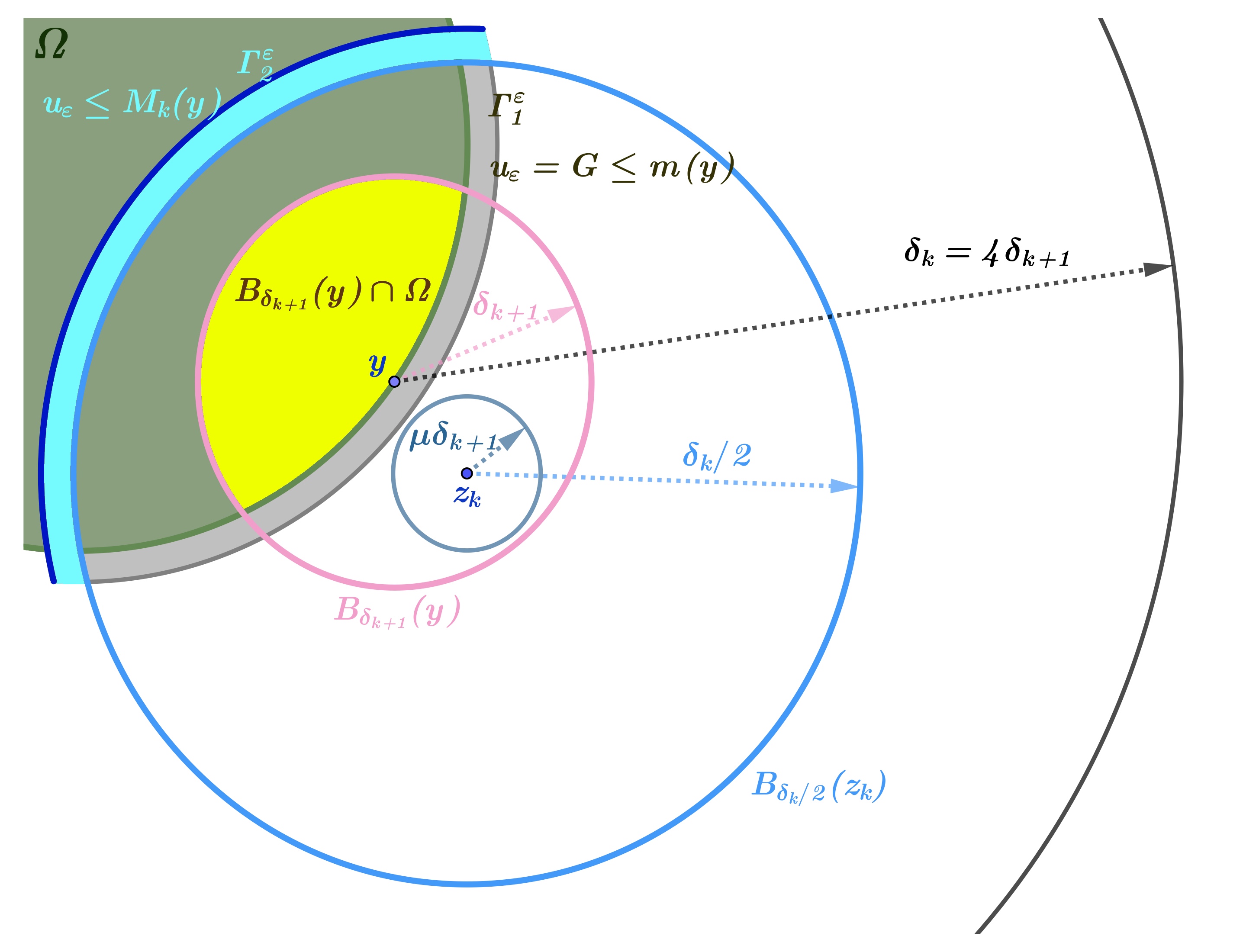}
\caption{Graphical explanation of the proof of Theorem \ref{thm:main2}}
\label{fig:convLipDom}
\end{figure}

Next, we write down the explicit results of the iteration that follow in a standard way.

\begin{corollary}\label{TheoremBoundary}
Given $\eta>0$, there exist $\delta=\delta(\eta, G, \bar{\delta})$, $k_0=k_0(\eta, \mu, p, G)$, $\eps_0=\eps_0(\eta, \delta,\mu, k_0)$ such that
\begin{equation*}
|u_\eps(x)-G(y)|\leq \frac{\eta}{2},
\end{equation*}
for all $y\in\partial\Omega$, $x\in B_{\delta/4^{k_0}}(y)\cap\Omega$ and $\eps\leq\eps_0$.
\end{corollary}
For example, if  $\osc_{O} G=\sup_{O}{G}-\inf_{O}{G}$ we can take
$$k_0=\left[\log_{\theta}\left(\frac{\eta}{4 \osc_{O} G}\right)\right]+1,$$
where $[\,\cdot\,]$ denotes the integer part function and $\theta$ is as in (\ref{eq:theta}).

It follows now that $\vv$ and $\underline{v}$ defined in \eqref{barles} satisfy
\begin{equation}\label{boundarycomparison}
\limsup_{\Omega \ni x\to y}\vv(x) \le G(y)\le \liminf_{\Omega\ni x\to y}\underline{v}(x)
\end{equation}
for all $y\in\partial\Omega$. By  the same argument as in the proof of Theorem \ref{thm:main} we  conclude that 
$\vv$ is a viscosity $p$-subsolution and $\underline{v}$ is a viscosity $p$-supersolution, since \eqref{boundarycomparison} ensures that the boundary condition is taken in the usual sense. Again, by construction $\overline{v}\leq \underline{v}$, and by the comparison principle for viscosity solutions given by Theorem \ref{thm:clasicalcompprinc} we have  that $\underline{v}\leq \overline{v}$. Thus, we get  $\underline{v}=\overline{v}=\lim_{\eps \to 0} u_\eps$.

\section{Dynamic programming principles with continuous solutions}\label{sec:contDPP}
Another family of dynamic programming principles can be formulated as
\begin{equation}\label{eq:DPPgen}
u_\eps(x)=\delta^\eps(x) A_\eps [u_\eps] (x) +(1-\delta^\eps(x)) G(x) \quad \textup{for} \quad x\in \Omega_E
\end{equation}
where
\begin{equation}\label{eq:w2}
\delta^\eps(x)= \left\{
\begin{array}{cccl}
1 &\text{ if }&x\in \Omega\setminus I_\eps\\
\displaystyle\frac{1}{\eps}d(x,\partial\Omega)&\text{ if }&x\in I_\eps\\
0 & \text{ if }&x\in O.
\end{array}
\right.
\end{equation}
The most interesting property of the dynamic programming principle \eqref{eq:DPPgen} is that in some cases they are known to produce continuous solutions (cf. \cite{LPS14, Har16}). 
\begin{remark}
Note that we can recover the dynamic programming principle \eqref{eq:expDPP1} from \eqref{eq:DPPgen} by choosing
 \begin{equation*}
\delta^\eps(x)= \left\{
\begin{array}{cccl}
1 &\text{ in }& \Omega\\
0 & \text{ in }& O.
\end{array}
\right. 
\end{equation*}
\end{remark}
For clarity of the presentation we restate here the assumptions \eqref{as:stab}, \eqref{as:stab} and \eqref{as:comppr} in the context of  the dynamic programming principle \eqref{eq:DPPgen} as well as the corresponding main results.
\begin{equation}\label{as:Hstab}
\begin{split}
&\textup{For all $ \eps>0$ there exists $ u_\eps \in B(\Omega_E)$ solution of \eqref{eq:DPPgen}
with a}\\ &\textup{bound on $\|u_\eps\|_{L^\infty(\Omega_E)}$ uniform in $\eps$.}
\end{split}
\end{equation}
\begin{equation}\label{as:Hstab2}
\begin{split}
&\textup{For all $ \eps>0$ there exists $ u_\eps \in B(\Omega_E)$ solution of \eqref{eq:DPPgen},  and }\\
&\inf_{O} G \leq u_\eps(x) \leq \sup_{O}{G} \quad \textup{for all}\quad x\in \Omega.
\end{split}
\end{equation}
\begin{equation}\label{as:Hcomppr}
\begin{split}
&\textup{Let $u^1_\eps$ and $u^2_\eps$ be a subsolution and a supersolution of \eqref{eq:DPPgen} with boundary}\\ &\textup{data $G_1$ and $G_2$ respectively. If $G^1\leq G^2$ on $O$ then $u_\eps^1\leq u_\eps^2$ in $\Omega_E$.} 
\end{split}
\end{equation}
\begin{theorem}\label{thm:main3}
Assume $p\in(1,\infty]$, $\Omega\subset \R^d$ be a bounded $C^2$-domain and $g\in C(\partial\Omega)$. Let $\{A_\eps\}_{\eps>0}$ be a mean value property for the $p$-Laplacian.  Let also $\{u_\eps\}_{\eps>0}$ be a sequence of solutions of the corresponding \eqref{eq:DPPgen}  satisfying  the assumption \eqref{as:Hstab}. Then we have that 
\[
\textup{$u_\eps\to v$  uniformly in $\overline{\Omega}$ as $\eps\to0$,  where $v$ is the unique viscosity solution of \eqref{dproblem}.}
\]
\end{theorem}
\begin{theorem}\label{thm:main4}
Let $p\in(1,\infty]$, $\Omega\subset \R^d$ be a bounded Lipschitz domain and $g\in C(\partial\Omega)$. Let $\{A_\eps\}_{\eps>0}$ be a mean value property for the $p$-Laplacian.  Let also $\{u_\eps\}_{\eps>0}$ be a sequence of solutions of the corresponding \eqref{eq:DPPgen}  satisfying  the assumptions \eqref{as:Hstab2} and \eqref{as:Hcomppr}. Then we have that 
\[
\textup{$u_\eps\to v$  uniformly in $\overline{\Omega}$ as $\eps\to0$,  where $v$ is the unique viscosity solution of \eqref{dproblem}.}
\]
\end{theorem}

At this point, it is clear that the outlines of the proof of Theorem \ref{thm:main3} and Theorem \ref{thm:main4} are exactly as in Theorem \ref{thm:main} and Theorem \ref{thm:main2} and we skip the unnecessary details. The only nontrivial adaptation is to show that the scheme associated to \eqref{eq:DPPgen} is monotone and stable. 

More precisely,  the dynamic programming principle  \eqref{eq:DPPgen} can be written as
\begin{equation}
\left\{
\begin{array}{cccl}
u_\eps(x) & = & A_\eps [u_\eps] (x) & \text{ in } \Omega\setminus I_\eps\\
u_\eps(x) & = & \delta^\eps (x)A_\eps [u_\eps] (x)  + (1-\delta^\eps(x))G(x)& \text{ in }  I_\eps\\
u_\eps(x)&=& G(x) & \text{ on } {O}.
\end{array}
\right.
\end{equation}
which leads to the following scheme defined for smooth functions as
\begin{equation}\label{eq:DPPscheme2}
S(\eps, x, \phi(x), \phi)=\left\{\begin{array}{cccl}
\frac{1}{c_{p,d} \eps^2}\left(\phi(x)-A_\eps[\phi](x)\right)& \text{ if } &x\in\Omega\\
\phi(x) - \delta^\eps (x)A_\eps [\phi] (x)  - (1-\delta^\eps(x))G(x)& \text{ if } &x\in  I_\eps\\
\phi(x)-G(x) & \text{ if } &x\in {O}
\end{array}\right.
\end{equation}

Then,  the dynamic programming principle \eqref{eq:DPPgen} can be formulated as
\begin{equation}\label{eq:DPP4}
S(\eps, x, u_\eps(x), u_\eps)=0 \quad \textup{for all} \quad x\in \Omega_E.
\end{equation}

We have the following monotonicity result:
\begin{lemma}\label{lem:mono2}
Let $\Omega \subset \R^d$ be a bounded domain and $A_\eps$ be an average. Then $S$ defined by \eqref{eq:DPPscheme2} is monotone, that is,
for all $\eps>0$, $x\in \Omega_E$,  $t\in \R$ and $u,v\in B(\Omega_E)$ such that $u\leq v$ we have that 
\[
S(\eps, x, t, v)\leq S(\eps, x, t, u).
\]
\end{lemma}
\begin{proof}
If $x\in \Omega\setminus \I_\eps$ or $x\in O$, everything is as in the proof proof Lemma \ref{lem:mono} for the scheme \eqref{eq:DPPscheme}. On the other hand,  since $0\leq \delta^\eps\leq1$,  if $x\in I_\eps$, then we have 
\begin{equation*}
\begin{split}
S(\eps, x, t, v)&=t - \delta^\eps (x)A_\eps [v] (x)  - (1-\delta^\eps(x))G(x)\\
&\leq t  - \delta^\eps (x)A_\eps [u] (x)  - (1-\delta^\eps(x))G(x) \leq S(\eps, x, t, u).\qedhere
\end{split}
\end{equation*}
\end{proof}
For consistency, we have to work a little bit more, but the result remains the same:

\begin{lemma}\label{lem:cons2}
Let $\{A_\eps\}_{\eps>0}$ be a  mean value property for the $p$-Laplacian and let $S$ be given by \eqref{eq:DPPscheme2}. Then, for all $x\in \overline{\Omega}$ and $\phi\in C_b^\infty(\Omega_{E})$ such that $\nabla\phi(x)\not=0$, we have that
\[
\limsup_{\eps\to0,\ y\to x,\ \xi \to 0} S(\eps, y, \phi(y)+\xi, \phi+\xi)\leq \left\{\begin{array}{cccl}
-\Delta_p^N\phi(x)& \text{ if } &x\in\Omega\\
\max\{-\Delta_p^N\phi(x), \phi(x)-G(x)\}& \text{ if } &x\in \partial{\Omega}
\end{array}\right.
\]
and 
\[
\liminf_{\eps\to0,\ y\to x,\ \xi \to 0} S(\eps, y, \phi(y)+\xi, \phi+\xi)\geq \left\{\begin{array}{cccl}
-\Delta_p^N\phi(x)& \text{ if } &x\in\Omega\\
\min\{-\Delta_p^N\phi(x), \phi(x)-G(x)\}& \text{ if } &x\in \partial{\Omega}.
\end{array}\right.
\]
\end{lemma} 
\begin{proof}
As before in the proof of Lemma \ref{lem:cons}, we can assume $\nabla\phi\not=0$ in a neighbourhood of $x$ by regularity of $\phi$.
If $x\in \Omega$, take $\eps< \rho:=d(x,\partial \Omega)/2$. Then $B:=B_\rho (x) \subset \Omega \setminus I_\eps$, and thus, using the same idea of the proof of Lemma \ref{lem:cons},
\begin{equation*}
\begin{split}
\limsup_{\eps\to0,\ y\to x,\ \xi \to 0} S(\eps, y, \phi(y)+\xi, \phi+\xi)&= \limsup_{\eps\to0,B\ni y\to x,\ \xi \to 0} S(\eps, y, \phi(y)+\xi, \phi+\xi)\\
&=\limsup_{\eps\to0,B\ni y\to x} \frac{1}{c_{p,d} \eps^2}\left(\phi(y)-A_\eps[\phi](y)\right)= -\Delta_p^N \phi(x).
\end{split}
\end{equation*}
Next, let $x\in \partial \Omega$. We can approach $x$ from points $y\in \Omega\setminus I_\eps$, $y\in I_\eps$ and $y\in O$. We have to treat the cases of $\limsup$ and $\liminf$ separately since the argument is slightly different. In both cases, two of the three limits follow as in the proof of Lemma \ref{lem:cons}, so we skip the details for them:
\begin{equation*}
\begin{split}
\limsup_{\eps\to0,\ y\to x,\ \xi \to 0} &S(\eps, y, \phi(y)+\xi, \phi+\xi)
= \max \bigg\{-\Delta_p \phi(x),\phi(x)-G(x), \\
&  \limsup_{\eps\to0,I_\eps \ni y\to x, \xi \to0} \left( \phi(y) - \delta^\eps (y)A_\eps [\phi] (y)  - (1-\delta^\eps(y))G(y)+(1-\delta^\eps(y))\xi \right) \bigg\}\\
= &\max \bigg\{-\Delta_p \phi(x),\phi(x)-G(x), \\
  \limsup_{\eps\to0,I_\eps \ni y\to x, \xi \to0} &\left(  \delta^\eps (y)(\phi(y)-A_\eps [\phi] (y))  + (1-\delta^\eps(y))(\phi(y)-G(y))+(1-\delta^\eps(y))\xi \right) \bigg\}.\\
\end{split}
\end{equation*}
Note that,
\[
 \delta^\eps (y)(\phi(y)-A_\eps [\phi] (y))  \leq \eps^2 \|-\Delta_p\phi\|_{L^\infty}  \delta_\eps^2\leq \eps^2 \|-\Delta_p\phi\|_{L^\infty} \to0 \quad \textup{as} \quad \eps \to0
\]
and $(1-\delta^\eps(y))\xi   \leq |\xi|\to 0$ as $\xi \to0$. On the other hand, since $\phi$ and $G$ are uniformly continuous continuous, there exist a modulus of continuity $\Lambda_{G,\phi}$ such that
\[
\phi(y)-G (y)\leq \phi(x)-G (x) + \Lambda_{G,\phi}(|x-y|).
\]
Note also that, 
\[
 \limsup_{\eps\to0,I_\eps \ni y\to x } (1-\delta^\eps(y))=1
\]
since it is trivially smaller than one, and the choice $y=x$ has $1$ as limit. Then
\[
\begin{split}
\limsup_{\eps\to0,I_\eps \ni y\to x }& ( 1-\delta^\eps) (y)(\phi(y)-G (y))\leq \limsup_{\eps\to0,I_\eps \ni y\to x} (1- \delta^\eps (y))(\phi(x)-G (x)+\Lambda_{G,\phi}(|x-y|) )\\
&\leq (\phi(x)-G (x))  \limsup_{\eps\to0,I_\eps \ni y\to x} (1- \delta^\eps (y))  +  \limsup_{I_\eps \ni y\to x} \Lambda_{G,\phi}(|x-y|) \\
&= \phi(x)-G (x),
\end{split}
\]
and we conclude that 
\[
\limsup_{\eps\to0,\ y\to x,\ \xi \to 0} S(\eps, y, \phi(y)+\xi, \phi+\xi)\leq  \max \left\{-\Delta_p \phi(x),\phi(x)-G(x)\right\}.
\]
In the case of the $\liminf$ we aggregate the terms in a different way, arriving to
\begin{equation*}
\begin{split}
\liminf_{\eps\to0,\ y\to x,\ \xi \to 0} &S(\eps, y, \phi(y)+\xi, \phi+\xi)
= \min \bigg\{-\Delta_p \phi(x),\phi(x)-G(x), \\
&   \liminf_{\eps\to0,I_\eps \ni y\to x, \xi \to0} \left( \phi(y)-G(y) - \delta^\eps (y)A_\eps [\phi] (y)  +\delta^\eps(y)G(y)+(1-\delta^\eps(y))\xi \right) \bigg\}\\
\end{split}
\end{equation*}
Note that 
\[
-\delta^\eps (y)A_\eps [\phi] (y) \geq- \delta^\eps (y) \|\phi\|_{L^\infty}, \quad \delta^\eps(y)G(y)\geq - \delta^\eps(y)\|G\|_{L^\infty} \quad \textup{and} \quad (1-\delta^\eps(y))\xi  \geq- |\xi|.
\]
But we now have that 
\[
\liminf_{\eps\to0,I_\eps \ni y\to x} \delta^\eps (y)=0.
\]
Then
\begin{equation*}
\begin{split}
 &\liminf_{\eps\to0,I_\eps \ni y\to x, \xi \to0} \left( \phi(y)-G(y) - \delta^\eps (y)A_\eps [\phi] (y)  +\delta^\eps(y)G(y)+(1-\delta^\eps(y))\xi \right) \\
&\geq \phi(x)-G(x) +  \liminf_{\eps\to0,I_\eps \ni y\to x, \xi \to0} ( -\Lambda_{G,\phi}(|x-y|)-\delta^\eps (y) \|\phi\|_{L^\infty} - \delta^\eps(y)\|G\|_{L^\infty}- |\xi|)\\
&= \phi(x)-G(x),
\end{split}
\end{equation*}
and we conclude that 
\[
\liminf_{\eps\to0,\ y\to x,\ \xi \to 0} S(\eps, y, \phi(y)+\xi, \phi+\xi)\geq  \min \left\{-\Delta_p \phi(x),\phi(x)-G(x)\right\}.\qedhere
\]
\end{proof}

\section{Some representative examples included in our theory}\label{sec:examples}
In this section we present some examples of dynamic programming principles for the $p$-Laplacian that naturally fall  into our general framework. Some of them have been extensively studied in the literature, as we will comment below.

We recall that $\Omega\subset \R^d$ is a bounded Lipschitz domain,   $\{A_\eps\}_{\eps>0}$ a mean value property (in the sense of Definition \ref{def:MVP}) and $G$ is a continuous extension of $g\in C(\partial \Omega)$ to $\Omega_E$. We have considered two families of dynamic programming principles:
\begin{equation}\label{eq:DPP6a}
\left\{
\begin{array}{cccl}
u_\eps(x) & = & A_\eps [u_\eps] (x) & \text{ in } \Omega,\\
u_\eps(x)&=& G(x) & \text{ on } {O}.
\end{array}
\right.
\end{equation}
and
\begin{equation}
\label{eq:DPP6b}
u_\eps(x)=\delta^\eps(x) A_\eps [u_\eps] (x) +(1-\delta^\eps(x)) G(x) \quad \textup{in}\quad \Omega_E,
\end{equation}
with 
\begin{equation*}
\delta^\eps(x)= \left\{
\begin{array}{cccl}
1 &\text{ if }&x\in \Omega\setminus I_\eps\\
\displaystyle\frac{1}{\eps}d(x,\partial\Omega)&\text{ if }&x\in I_\eps\\
0 & \text{ if }&x\in O.
\end{array}
\right.
\end{equation*}
We will present a series of mean value properties for the $p$-Laplacian with $p\in(1,\infty]$ and comment the literature on them for dynamic programming principles of the form \eqref{eq:DPP6a} and \eqref{eq:DPP6b}. 

\subsection{Dynamic programming principles in $\R^d$}

We will always assume that $\eps>0$, $\phi\in C^\infty_b(\Omega_E)$ with $\nabla\phi\not=0$, and $A[\phi]:\Omega\to \R$. We will denote by $B_\eps(x)$ to the ball of radius $\eps$ centred at some $x\in  \R^d$, and write just $B_\eps$ if it is centred at $x=0$.
\subsubsection{Case $p=2$: The Laplacian}\label{sec:MVPp2}
The most iconic case is the so-called mean value property for $\Delta$, given by
\begin{equation*}
A_\eps^2[\phi](x)=\frac{1}{|B_\eps|} \int_{B_\eps(x)} \phi(y) dy.
\end{equation*}
Clearly $A_\eps^2$ is an average for fixed $\eps$, and the family $\{A_\eps^2\}_{\eps>0}$ is a mean value property for the Laplacian: Indeed, using the Taylor theorem for $\phi$ up to order four, integrating by $|B_\eps|^{-1}\int_{B_\eps(x)}$ and using the symmetry of the operator and domain allow to conclude that
\[
\phi(x)=A_\eps^2[\phi](x)+\frac{\eps^2}{2(2+d)} (-\Delta\phi(x))+o(\eps^2).
\]
\begin{remark}
One can replace $B_\eps$ by other symmetric domains like e.g.\ squares centred with sides of length $2\eps$ or spheres $\partial B_\eps(x)$. 
\end{remark}

\subsubsection{Case $p=\infty$: The $\infty$-Laplacian}\label{sec:MVPinfty} The following mean value property was introduced by Le Gruyer and Archer in \cite{LA98}:
\begin{equation}\label{eq:MVPinfty}
A_\eps^\infty[\phi](x)=\frac{1}{2}\sup_{B_\eps(x)} \phi+\frac{1}{2}\inf_{B_\eps(x)} \phi.
\end{equation}
Trivially $A_\eps^\infty$ is an average for fixed $\eps$. It is now standard to show that  family $\{A_\eps^\infty\}_{\eps>0}$ is a mean value property for the $\infty$-Laplacian. Recall that we have the expansion
\begin{equation*}
\begin{split}
A^{\infty}_\eps[\phi](x)&= \frac{1}{2}\phi\left(x+\eps\frac{\nabla\phi(x)}{|\nabla\phi(x)|}\right)+ \frac{1}{2}\phi\left(x-\eps\frac{\nabla\phi(x)}{|\nabla\phi(x)|}\right)+o(\eps^2)\\
&=u(x)+\frac{\eps^2}{2}\Delta_\infty^N\phi(x)+o(\eps^2).
\end{split}
\end{equation*}
The properties \eqref{as:stab2} and \eqref{as:comppr} for the dynamic programming principle \eqref{eq:DPP6a} with $A_\eps^\infty$ given by \eqref{eq:MVPinfty} are shown in  e.g \cite{LA98} by analytical methods  and in \cite{PSSW09} using tug-of-war games.

\subsubsection{Case $p\in(2,\infty)$: The $p$-Laplacian}
The following mean value property was introduced by Manfredi, Parviainen and Rossi in \cite{MPR10}:
\begin{equation}\label{eq:MVPpLap}
A_\eps^p[\phi](x)=\alpha_p\left(\frac{1}{2}\sup_{B_\eps(x)} \phi+\frac{1}{2}\inf_{B_\eps(x)} \phi \right)+\beta_p \frac{1}{|B_\eps|} \int_{B_\eps(x)} \phi(y) dy
\end{equation}
with  $\alpha_p=\frac{p-2}{p+d}$ and $\beta_p=\frac{2+d}{p+d}$ ($\alpha_p+\beta_p=1$). Note that 
\begin{equation}\label{eq:interpMPV}
A_\eps^p[\phi](x)=\alpha_pA^{\infty}_\eps[\phi]+ \beta_pA_\eps^2[\phi](x)
\end{equation}
so  $A_\eps^p$ is a convex combination of averages, and thus is itself an average for a fixed $\eps$. Property \eqref{eq:interpMPV} together with the interpolation formula \eqref{eq:interp2} also ensures that $\{A_\eps^\infty\}_{\eps>0}$ is a mean value property for the $p$-Laplacian:
\begin{equation*}
\begin{split}
A_\eps^p[\phi](x)&=\alpha_p\left(\phi(x)+\frac{\eps^2}{2}\Delta_\infty^N\phi(x)\right)+\beta_p\left(\phi(x)+\frac{\eps^2}{2(2+d)}\Delta\phi(x)\right)+o(\eps^2)\\
&=u(x)+\frac{\eps^2p}{2(p+d)} \underbrace{\left(\frac{p-2}{p}\Delta_\infty^N\phi(x)+ \frac{1}{p}\Delta \phi(x)\right)}_{\Delta_p^N\phi(x)}+o(\eps^2)
\end{split}
\end{equation*}
The properties \eqref{as:stab2} and \eqref{as:comppr} for the dynamic programming principle \eqref{eq:DPP6a} with $A_\eps^p$ given by \eqref{eq:MVPpLap} are shown in  \cite{MPR10} by probabilistic methods and in \cite{LPS14} with  analytic tools. $A_\eps^p$ also been studied for the regularized version \eqref{eq:DPPgen} in \cite{AHP17}.
\subsubsection{Case $p=\infty$: The $\infty$-Laplacian via $p$-Laplacians as $p\to\infty$.} Consider the family of averages given by
\begin{equation}\label{eq:MVPinftyLap2}
A_\eps^{p\to\infty}[\phi](x)=(1-\eps^3)\left(\frac{1}{2}\sup_{B_\eps(x)} \phi+\frac{1}{2}\inf_{B_\eps(x)} \phi \right) +\eps^3 \frac{1}{|B_\eps|} \int_{B_\eps(x)} \phi(y) dy.
\end{equation}
Note that for fixed $\eps$,  $A_\eps^{p\to\infty}$ coincides with $A_\eps^p$ given by \eqref{eq:MVPpLap} for a certain choice of $p$ depending on $\eps$. As a consequence, properties \eqref{as:stab2} and \eqref{as:comppr} are also true for the dynamic programming principle \eqref{eq:DPP6a}. Moreover,
\begin{equation*}
\begin{split}
A_\eps^{p\to\infty}[\phi](x)&=(1-\eps^3)\left(\phi(x)+\frac{\eps^2}{2}\Delta_\infty^N\phi(x)\right)+\eps^3\left(\phi(x)+\frac{\eps^2}{2(2+d)}\Delta\phi(x)\right)+o(\eps^2)\\
&=u(x)+\frac{\eps^2}{2}\Delta_\infty^N\phi(x)+o(\eps^2),
\end{split}
\end{equation*}
thus, $A_\eps^{p\to\infty}$ is a mean value property of the $\infty$-Laplacian. More precisely, let $u_\eps$ be the solution of \eqref{eq:DPP6a} with the mean value property \eqref{eq:MVPinftyLap2} and $u$ be the viscosity solution of the Dirichlet problem for the $\infty$-Laplacian given by \eqref{dproblem}. Then $u_\eps\to u$ uniformly as $\eps\to0$. This strategy provides an alternative dynamic programming principle for the $\infty$-Laplacian.
\subsubsection{Case $p\in(1,\infty)$: The $p$-Laplacian} As introduced by \cite{KMP12}, consider the mean value property
\begin{equation}\label{eq:MVPpLap2}
A_\eps[\phi](x)=\frac{1}{2}\sup_{0<|\nu|\leq1}\Phi(x,\nu,\eps)+\frac{1}{2}\inf_{0<|\nu|\leq1}\Phi(x,\nu,\eps)
\end{equation}
with
\[
\Phi(x,\nu,\eps):=\frac{p-1}{p+d} \phi(x+\nu \eps)+\frac{1+d}{p+d} \int u(x+h) d\mu_{\nu}(h)
\]
where $\mu_\nu$ is the probability measure over the $(n-1)$-dimensional closed disk of radius $\eps$ orthogonal to
the vector $\nu$. The proof of the fact that 
\[
A_\eps^p[\phi](x)=u(x)+\frac{\eps^2p}{2(p+d)} \underbrace{\left(\frac{p-1}{p}\Delta_\infty^N\phi(x)+ \frac{1}{p}\Delta_1^N \phi(x)\right)}_{\Delta_p^N\phi(x)}+o(\eps^2)
\]
relies again on the Taylor theorem and also on the interpolation formula \eqref{eq:interp1}. We refer to \cite{KMP12} for the precise details.  The properties \eqref{as:stab2} and \eqref{as:comppr} for the dynamic programming principles \eqref{eq:DPP6a} and  \eqref{eq:DPP6b} with $A_\eps$ given by \eqref{eq:MVPpLap2} have also been presented  in \cite{Har16}. 
\subsection{Dynamic programming principles on lattices.} 
We present here dynamic programming principles that only act on a finite number of points. 

Given a discretization parameter $h>0$, consider the grid defined by
\[
\G_h:=h\Z^d= \{x_\alpha=h\alpha: \ \alpha\in \Z^d\}.
\]
We will also need the concept of discrete ball and discretized set $\mathcal{S}$ given by
\[
B_\eps^{h}(x)=(x+\G_h)\cap B_\eps(x)\quad \textup{and} \quad \mathcal{S}_h=\mathcal{S}\cap \G_h.
\]
We will denote by $|B_\eps^h(x)|$  the cardinality of the set $B_\eps^{h}(x)$.
An interesting fact is that that for all $x_\alpha \in \G_h$, then $x_\alpha+\G_h=\G_h$ and thus $B^h_\eps(x_\alpha)\subset\G_h$. This will allow for fully discrete dynamic programming principle, which will produce numerical methods for \eqref{dproblem} (as done for example by Oberman in \cite{Ob13}). More precisely, a numerical method for \eqref{dproblem} will take the form
\begin{equation}\label{eq:numDPP}
\left\{
\begin{array}{cccl}
u_{\eps,h}(x_\alpha) & = & A_{\eps,h} [u_{\eps,h}] (x_\alpha) & \text{ for } x_\alpha \in \Omega_h ,\\
u_{\eps,h}(x_\alpha)&=& G(x_\alpha) & \text{ for } {x_\alpha\in O_h},
\end{array}
\right.
\end{equation}
where now $A_\eps:B(\Omega_E^h)\to\R$ and $u_{\eps,h}:\Omega_E^h\to\R$ is a bounded function defined only on the discrete domain $\Omega_E^h$.

\subsubsection{Case $p\in[2,\infty]$: The $p$-Laplacian}\label{sec:MVPdiscrete} We can define the discrete counterpart of $A_\eps^p$ given by \eqref{eq:MVPpLap} as
\begin{equation}\label{eq:MVPdiscP}
A^{p}_{\eps,h}[\phi](x)=\alpha_p\left(\frac{1}{2}\max_{B_\eps^h(x)} \phi+\frac{1}{2}\min_{B_\eps^h(x)} \phi \right)+\beta_p \frac{1}{|B_\eps^h(x)|}\sum_{y \in B_\eps^h(x)} \phi(y).
\end{equation}
The properties \eqref{as:stab2} and \eqref{as:comppr} for the dynamic programming principle \eqref{eq:DPP6a} (and consequently for \eqref{eq:numDPP}) with $A_{\eps,h}^p$ follows exactly as for $A_\eps^p$ with the restriction $h=o(\eps^2)$. The fact that $A_{\eps,h}^p$ is also a mean value property for the $p$-Laplacian is also standard and we omit the details. We refer to \cite{CLM17} for precise details and applications of $A^{p}_{\eps,h}$ to the obstacle problem.

Note that the second term in \eqref{eq:MVPdiscP} is precisely the composed midpoint quadrature rule for $\int_{B_\eps}$ which is known to produce an error of order $O(h)$. To produce more precise numerical methods in the form of a dynamic programming principle,  one can consider higher order monotone quadrature rules that will have the form 
\[
\sum_{y\in B_{\eps}^h(x)} \phi(y)\omega(y)
\]
where $\omega(y)>0$. Precise form of the weights $\omega$ can be obtained for example from the Newton-Cotes formulas. They are known to produce positive weights (and thus monotone schemes) for orders up to $O(h^7)$.

\subsubsection{Case $p=2$: The Laplacian.}\label{sec:MVPFD} The best known discrete mean value property for the Laplacian is given by the   discrete Laplacian:
\[
\Delta_h\phi(x):=\sum_{i=1}^d \frac{\phi(x+e_ih)+\phi(x-e_ih)-2\phi(x)}{h^2}.
\]
From the Taylor theorem we see that $\Delta_h\phi(x)=\Delta \phi(x)+O(h^2)$. We naturally define 
\begin{equation}\label{eq:MVPdisLap}
A_h^2[\phi](x)=\frac{1}{2d}\sum_{i=1}^d\left(\phi(x+e_ih)+\phi(x-e_ih)\right).
\end{equation}
It is clear that $\{A^2_h\}_{h>0}$ is a mean value property for the Laplacian since
\[
\phi(x)=A_h^2[\phi](x) - \frac{h^2}{2 d}\Delta \phi(x)+ \underbrace{O(h^4)}_{o(h^2)}.
\]
The properties \eqref{as:stab2} and \eqref{as:comppr} for the dynamic programming principle \eqref{eq:DPP6a} (and consequently for \eqref{eq:numDPP}) with $A_{h}^p$ given by \eqref{eq:MVPdisLap} is standard in numerical analysis framework (see for example \cite{Ob13}).

\begin{remark}
As shown in \cite{Ob05}, it is also possible to construct monotone schemes (and consequently discrete dynamic programming principles) for the $\infty-Laplacian$ by means of the so-called absolutely minimizing Lipschitz extensions. As expected, one can combine them with $A_h^2$ given by \eqref{eq:MVPdisLap} to produce also discrete dynamic programming principles for the $p$-Laplacian for $p\in[2,\infty]$ (cf. \cite{Ob13}).
\end{remark}

\section*{Acknowledgements} 
F.d.T is supported by the Toppforsk (research excellence) project Waves and Nonlinear Phenomena (WaNP), grant no. 250070 from the Research Council of Norway and  by the grant PGC2018-094522-B-I00 from the MICINN of the Spanish Government. \normalcolor
M.P. is supported by the Academy of Finland project no.\ 298641.
\par
We are very thankful to E. R. Jakobsen for discussions on generalized viscosity solutions and for the reference \cite{J10}. \par
The first author wants to thank the University of Pittsburgh for hosting him during a research stay when part of this research was conducted.
Part of this research was also done while the second author visited NTNU and the University of Jyv\"askyl\"a. He thanks both institutions for their hospitality.\par

\end{document}